\documentclass[11pt,reqno]{article}
\usepackage{bbding}
\usepackage{pifont}
\usepackage{amsmath}
\usepackage{mathrsfs}
\usepackage{amsfonts,bm}
\usepackage{amsthm}
\usepackage{epsfig}
\usepackage[]{harpoon}
\usepackage[active]{srcltx}
\usepackage{indentfirst,latexsym}
\usepackage{psfrag}
\usepackage[]{amssymb}
\usepackage{cite}
\usepackage{caption}

\newcommand{\new}{\newcommand*}\new{\rnew}{\renewcommand*}
\new{\newe}{\newenvironment*}\new{\stl}{\setlength}
\stl{\textwidth}{155mm}\stl{\textheight}{22cm}\stl{\headheight}{0cm}
\stl{\topmargin}{0cm}\stl{\oddsidemargin}{0.5cm}\stl{\evensidemargin}{0cm}
\rnew{\arraystretch}{1.1}\rnew{\baselinestretch}{0.92}
\renewcommand{\thefootnote}{\ding{73}}
\newtheorem{thm}{Theorem}

\newtheorem{lem}{Lemma}[section]

\newtheorem{defn}{Definition}[section]
\newtheorem{rem}{Remark}

\newtheorem{asu}{Assumption}

\newcommand{\dps}{\displaystyle}

\newcommand{\fr}{\frac}

\newcommand{\pa}{\partial}

\numberwithin{equation}{section}
\newe{keywords}
   {\begin{quote}{\bf Keywords:}}
      {\end{quote}}
\newe{AMS}
   {\begin{quote}{\bf MSC 2020:}}
      {\end{quote}}
\newe{MSC}{\vspace*{5mm}
    {\noindent\it Mathematics Subject Classification(2020):}}{}
\new{\sect}[1]{\section{#1}\setcounter{equation}{0}
 \setcounter{thm}{0}\setcounter{lmm}{0}\setcounter{rmk}{0} }

\begin{document}
%% ==================== Title =======================================

\title{Existence and singularity formation for the supersonic expanding wave of radially symmetric non-isentropic compressible Euler equations}

\author{
Geng Chen\thanks{Department of Mathematics,
University of Kansas, Lawrence, KS,
66045 ({\tt gengchen@ku.edu}).}
\and
Faris A. El-Katri\thanks{Department of Mathematics,
University of Kansas, Lawrence, KS,
66045 ({\tt elkatri@ku.edu}).}
\and
Yanbo Hu\thanks{Department of Mathematics, Zhejiang University of Science and Technology, Hangzhou, 310023, PR China ({\tt yanbo.hu@hotmail.com}).}
}

\rnew{\thefootnote}{\fnsymbol{footnote}}

\date{}

\maketitle
%%========================= abstract ==================================
\begin{abstract}
This paper studies the existence and singularity formation of supersonic expanding waves for the radially symmetric non-isentropic compressible Euler equations of polytropic gases. We introduce a suitable pair of gradient variables to characterize the rarefaction and compression properties of the solutions. Based on their Riccati equations, we construct several useful invariant domains to
establish a series of priori estimates of solutions under some assumptions on the initial data.
We show that the solution is smooth in the characteristic triangle or quadrangle domain if both of these two gradient variables are non-negative at the initial time. On the other hand, when one of these two variables is very negative at some initial point, the solution forms a
singularity in finite time.
\end{abstract}

\begin{keywords}
Non-isentropic Euler equations, radially symmetry, singularity formation, supersonic wave, hyperbolic conservation laws.
\end{keywords}

\begin{AMS}
76N15, 35L65, 35L67.
\end{AMS}

%%$$$$$$$$$$$$$$$$$$$$$$$$$$$$$$$$ section 1 %$$$$$$$$$$$$$$$$$$$$$$$$$$$$$$$
\section{Introduction}\label{S1}

The radially symmetric non-isentropic compressible Euler equations reads that \cite{courant}
\begin{align}\label{1.1}
\begin{split}
(r^m\rho)_t+(r^m\rho u)_r=&\ 0, \\
(r^m\rho u)_t+(r^m\rho u^2)_r+ r^m p_r=&\ 0, \\
S_t+uS_r=&\ 0,
\end{split}
\end{align}
where $t\geq0$ and $r>0$ are the time-space independent variables, $(\rho, u, S)(r,t)$ are the unknown variables having their ordinary meaning: $\rho(r,t)$ is the density, $u(r,t)$ is the particle velocity and $S(r,t)$ is the entropy. The constant $m\geq1$ is an integer, particularity, system \eqref{1.1} with $m=1,2$ correspond to cylindrically and spherically symmetric flows, respectively. The pressure $p$ is the function of $\rho$ and $S$, which takes the following form for the polytropic gas
\begin{align}\label{1.2}
p=Ke^{\frac{S}{c_v}}\rho^\gamma,
\end{align}
where $K>0$ is a constant, the constant $c_v>0$ is the specific heat at constant volume, the constant $\gamma>0$ is the adiabatic exponent which lies between 1 and 3 for most gases \cite{courant, smoller}. When $S\equiv Const.$, system \eqref{1.1} reduces to the radially symmetric isentropic Euler equations
\begin{align}\label{1.3}
\begin{split}
(r^m\rho)_t+(r^m\rho u)_r=&\ 0, \\
(r^m\rho u)_t+(r^m\rho u^2)_r+ r^m (K\rho^\gamma)_r=&\ 0.
\end{split}
\end{align}

As one of the most fundamental physical models of nonlinear hyperbolic conservation laws, the
compressible Euler equations have been extensively studied over a long history. Its biggest feature is that, even when initial data are small and smooth, the classical solution may form gradient blowup in finite time. What kind of initial data make smooth solutions exist or form singularities is one of the central issues of nonlinear hyperbolic conservation laws. In the pioneer work \cite{lax2}, Lax provided a beautiful answer to the $2\times2$ reducible homogeneous strictly hyperbolic systems. Subsequently, the general $n\times n$ systems were discussed among others in \cite{Ali, Schen, John, Kong, LZK, Liu1}. These early works show that, for the initial data of small smooth perturbation near a constant state, the genuinely nonlinear characteristic field can develop singularity in finite time if initial compression exists.

However, the study on the corresponding theory of large data is still limited, even for the important models such as the compressible Euler equations. For the 1-d isentropic Euler equations (system \eqref{1.3} with $m=1$), the Lax's framework can be directly employed to obtain a complete dichotomy result when $\gamma\geq3$, that is, the finite time singularity forms if and only if the initial compression exists. Nevertheless, this result does not applicable to the most physical case $1<\gamma<3$ due to the degeneracy of the Riccati equations by the disappearance of density.
Thus, a suitable density lower bound estimate is first needed to establish to acquire the relevant conclusions. In \cite{CPZ}, Chen, Pan, and Zhu derived a time-dependent lower
bound on the density when initial data far from the vacuum. Based on this density lower bound, they attained a complete picture on the finite time singularity formation for 1-d isentropic Euler equations with $1<\gamma<3$. Moreover, a singularity formation result for the 1-d non-isentropic Euler equations with strong compression initial data was also verified in \cite{CPZ}, using earlier results given in \cite{G3,CYZ}.
The optimal time-dependent lower bound on density for 1-d isentropic and non-isentropic Euler equations was identified by Chen in \cite{G9, G10}. In \cite{CCZ}, Chen, Chen, and Zhu established a sufficient condition for the formation of singularities of solutions for the non-isentropic Euler equations with large initial data that allow a far-field vacuum. They also constructed a global continuous non-isentropic solution for initial data containing a weak compression. The formations of singularities of smooth solutions for the 1-d isentropic and non-isentropic relativistic Euler equations were discussed by Athanasiou, Bayles-Rea and Zhu \cite{A-Zhu1, A-Zhu2}.

For the radially symmetric isentropic Euler equations \eqref{1.3}, one key challenge is overcoming the difficulties caused by geometric source terms. A natural idea is to mimic the 1-d case to adopt the gradient of Riemann invariants as the new variables to derive the a priori $C^1$ estimates of solutions. There is no doubt that, due to the presence of geometric terms, such a programme leads the Riccati equations being non-homogeneous, which poses great trouble in constructing the desired invariant domain of solutions. By applying the Riccati system in \cite{CYZ}, Cai, Chen, and Wang \cite{CCW} analyzed the complicated non-homogeneous Riccati equations through a very tedious process to demonstrate the singularity formation of supersonic expanding wave for \eqref{1.3} with $1<\gamma<3$. Based on the idea that the stationary solution is neither rarefactive nor compressive, in a recent paper \cite{CEHS1}, we found a pair of accurate gradient variables (called rarefaction and compression characters) to obtain the desired homogeneous Riccati equations. Then we proved that, for supersonic expanding waves of \eqref{1.3} with $1<\gamma<3$, smooth solutions with rarefactive initial data exist global-in-time, while singularity forms in finite time when the initial data include strong compression somewhere. This pair of gradient variables was subsequently adopted to study the  singularity formation for the supersonic inward wave of \eqref{1.3} in \cite{CEHS2}.

There have also many other works on the singularity formation for multi-d compressible Euler equations under various assumptions on the initial data. In \cite{sideris}, Sideris
established some singularity formation results for compactly supported initial data. The framework of Sideris involves some averaged quantities to avoid
local analysis of solutions, which makes that it is difficult to provide specific information on the properties of breakdown. A geometric framework was introduced by Christodoulou \cite{Ch1} to probe the singularity formation for the relativistic Euler equations in multispace dimensions. In \cite{Ch2}, Christodoulou and Miao applied the geometric framework to discuss the shock formation of the compressible isentropic irrotational Euler equations with compactly supported initial data near a constant solution. The 2-d case with nonzero vorticity was studied by Luk and Speck \cite{Luk1}, also see the related papers \cite{Di, Luk2} etc. In addition, we refer the reader to a series of papers \cite{Vicol0, Vicol1, Vicol2, Vicol3, Vicol5, Vicol4, SV, Yin} on the contribution of constructing shocks or pre-shocks.

The global existence of smooth solutions to the compressible Euler equations has attracted many attention as well. The complete existence results of the 1-d compressible isentropic Euler equations with rarefaction initial data can be found in \cite{lax2, CCZ, CPZ}. Partial results on the non-isentropic equations were presented in \cite{CCZ, CPZ, CY, Zhu}. Grassin \cite{Grassin} verified the global existence of smooth solutions to the multi-d compressible isentropic Euler equations under some assumptions of smallness and smoothness initial data. Based on the affine motions constructed by Sideris \cite{sideris1}, the global existence of near-affine solutions of multi-d compressible isentropic and non-isentropic Euler equations were established in among others \cite{HJ, Rickard1, Rickard2, SS}. When initial data are a small smooth perturbation near a constant state, Godin \cite{Godin} analyzed the lifespan of smooth solutions for the spherically symmetric non-isentropic Euler equations \eqref{1.1}. In \cite{Lai}, Lai and Zhu obtained the global existence of smooth solutions to the 2-d axisymmetric Euler equations for a class of initial data containing constant states near the origin.

In this paper, we focus on the existence and singularity formation of supersonic expanding wave for radially symmetric non-isentropic Euler equations \eqref{1.1} with \eqref{1.2}. A solution of \eqref{1.1} is called a supersonic expanding wave if it satisfies $u(r,t)>\sqrt{p_\rho(r,t)}>0$ at every point $(r,t)$ in the domain.
Compared with the isentropic case, the main challenge currently is dealing with the terms brought by the entropy function. Overcoming the effect of varying entropy is highly nontrivial, especially for multi-d solutions of Euler equations.

For the isentropic Euler equations \eqref{1.3}, we recently constructed in \cite{CEHS1} gradient variables by differentiating some special function that takes constant value in the stationary solution. The underlying idea stems from the observation that the stationary solution is neither rarefactive nor compressive; in other words, the chosen gradient variables should vanish in the stationary regime. Guided by this principle, we introduced appropriate gradient variables to characterize the rarefaction and compression properties of solutions. Owing to the special structure of these gradient variables, the corresponding Riccati equations are homogeneous, which plays an important in the analysis carried out in our previous work \cite{CEHS1}.

For the non-isentropic system \eqref{1.1} considered in the present work, we first find that the quantity $r^m \rho u$ is constant in stationary solutions. With this quantity as a starting point,
we construct the admissible gradient variables $(\alpha,\beta)$ as the following forms:
\begin{align}\label{1.4}
\alpha=-\frac{\partial_1(r^m \rho u)}{r^m \rho c_3}, \qquad \beta=-\frac{\partial_3(r^m \rho u)}{r^m \rho c_1},
\end{align}
where
\begin{align}\label{1.5}
c_1 =u-\sqrt{p_\rho(\rho, S)},\ \ c_3 =u+\sqrt{p_\rho(\rho, S)},\quad \partial_1=\partial_t+c_1 \partial_r, \ \ \partial_3=\partial_t+c_3 \partial_r.
\end{align}
Notably, the presence of entropy is clearly reflected in the structure of $(\alpha,\beta)$.
In particular, in contrast to the isentropic setting, the Riccati equations of $(\alpha,\beta)$ in the present non-isentropic framework are nonhomogeneous, a direct consequence of the spatially varying entropy. To dispose of the quantities $S_r$ and $S_{rr}$ appearing in the non-homogeneous terms, we reformulate them in Lagrangian coordinates, even though this renders the Riccati equations more complicated. Then the terms involving $S$ can be estimated and the invariant domains of the solution itself and gradient variables can be established under some suitable assumptions on the initial data. Especially, the supersonic expanding property of the solution can be preserved throughout the  solving domain under these initial conditions.
Furthermore, we are fortunate to find that all entropy-related coefficients are multiplied by the sound speed $\sqrt{p_\rho(\rho, S)}$ in the governing equations. This key feature enables us to establish a positive lower bound for the density and for some certain specially weighted gradient variables, which are crucial for us to complete the analysis.

\begin{defn}\label{def1}
Consider a $C^1$ solution of \eqref{1.1}. At any point $(x,t)$, we define the local rarefaction/compression character as
\begin{itemize}
\item In the $1$-character direction, the solution is rarefaction if $\beta>0$, compression if $\beta<0$.
\item  In the $3$-character direction, the solution is rarefaction if $\alpha>0$, compression if $\beta<0$.
\end{itemize}
\end{defn}

In this paper, we may slightly abuse the notation to say the solution is rarefactive at some point when both $\alpha\geq0$ and $\beta\geq0$.

The main results of the paper can be roughly summarized as follows.

\begin{itemize}

\item Let initial data $(\rho_0, u_0, S_0)(r)$ satisfy Assumptions \ref{asu1}-\ref{asu3} on the interval $[b_1,b_2]$ with $b_1>0$. Moreover, $\alpha(r,0)\geq0$ and $\beta(r,0)\geq0$ on $[b_1,b_2]$. Then the solution is smooth in the characteristic triangle or quadrangle domain generated by interval $[b_1,b_2]$.(Theorem \ref{thm1})

\item Let initial data $(\rho_0, u_0, S_0)(r)$ satisfy Assumptions \ref{asu1} and \ref{asu2} on the interval $[b_1,b_2]$ with $b_1>0$. If $\alpha(r,0)$ or $\beta(r,0)$ is very negative at some point on $[b_1,b_2]$, then the solution forms a singularity in finite time.(Theorem \ref{thm2})
\end{itemize}

The Assumptions \ref{asu1}-\ref{asu3} and the precise statements of these conclusions are presented in Section \ref{S22} below. It is well known that the key to proving the existence or singularity formation of smooth solution is to establish the $C^1$-estimates of the solution. These estimates come from the invariant domains of the solution, which rely on the non-homogeneous Riccati equations of $(\alpha,\beta)$.

These two results together indicate that our definition of the local rarefaction and compression character in Definition \ref{def1} is the {\it{correct}} definition in line with the rarefactive and compressive phenomena for the gas dynamics.

We note the first result on the large time existence of rarefactive non-isentropic supersonic expanding solutions is much more surprising than the second result on the singularity formation,  showing the strong initial compression will produce a singularity in finite time.
 Even for 1-d non-isentropic solutions, there are only some very special large time classical existence examples. For example, see \cite{CCZ}. In this paper, we find a large class of classical radial symmetric solutions for Euler equations.

The rest of the paper is organized as follows.
In Section \ref{S2}, we derive the characteristic equations in terms of the Riemann variables, and then present a detailed statement of the assumptions and main conclusions of the paper.
In Section \ref{S3}, we derive the Riccati equations of the gradient variables $(\alpha, \beta)$.
In Section \ref{S4}, we use the idea of invariant domain and the Riccati equations of $(\alpha, \beta)$ to establish the a priori $C^1$ estimate of the solution, including a pair of special weighted estimates of $(\alpha, \beta)$. Finally, in Section \ref{S5}, we are based on the previous estimates to complete the proof of the theorems of the paper. The global existence and singularity formation of smooth solution are demonstrated in Sections \ref{S51} and \ref{S52}, respectively.

\section{The equations, assumptions and main results}\label{S2}

In this section, we derive the characteristic equations of \eqref{1.1} and simplify the gradient variables $(\alpha,\beta)$ defined in \eqref{1.4} by introducing the Riemann variables. Then we provide the assumptions and main results of the paper in details.

\subsection{The characteristic equations}\label{S21}

Let $\gamma>1$. We introduce the following new unknown variable $h$ to take place of $\rho$
$$
h=\sqrt{p_\rho}=\sqrt{K\gamma}\,e^{\frac{S}{2c_v}}\rho^\frac{\gamma-1}{2}.
$$
or
\begin{align}\label{a1}
\rho=\gamma_k e^{-\fr{S}{c_\gamma}}h^{\fr{2}{\gamma-1}},
\end{align}
where
$$
\gamma_k=(K\gamma)^{\fr{1}{1-\gamma}},\quad c_\gamma=c_v(\gamma-1).
$$
The quantity $h$ is called local sound speed in gas dynamics. For smooth solutions, we can rewrite system \eqref{1.1} as
\begin{align}\label{2.1}
\begin{split}
h_t+uh_r+\fr{\gamma-1}{2}hu_r+\fr{\gamma-1}{2}\frac{m uh}{r}=&\ 0, \\
u_t+uu_r+\fr{2}{\gamma-1}hh_r-\fr{h^2}{\gamma c_\gamma}S_r=&\ 0, \\
S_t+uS_r=&\ 0.
\end{split}
\end{align}
System \eqref{2.1} can be written in the matrix form
$$
\textbf{U}_t+\textbf{A}\textbf{U}_r=\textbf{F},
$$
where
\begin{align*}
\textbf{U}=
\left(
\begin{array}{c}
h \\
u \\
S
\end{array}
\right),\quad
\textbf{A}=
\left(
\begin{array}{ccc}
u  & \fr{\gamma-1}{2}h  &  0 \\
\fr{2}{\gamma-1}h  &  u  &  -\fr{1}{\gamma c_\gamma}h^2 \\
0  &  0  &  u
\end{array}
\right),\quad
\textbf{F}=
-\fr{(\gamma-1)muh}{2r} \left(
\begin{array}{c}
1 \\
0 \\
0
\end{array}
\right).
\end{align*}
The three eigenvalues of $\textbf{A}$ are
\begin{align}\label{2.2}
c_1=u-h,\quad c_2=u, \quad c_3=u+h.
\end{align}

We then introduce the Riemann variables
\begin{align}\label{2.3}
w=u+\frac{2}{\gamma-1}h,\quad z=u-\frac{2}{\gamma-1}h,
\end{align}
and the directional derivatives
\begin{align}\label{2.4}
\pa_i=\pa_t+c_i\pa_r,\ \ (i=1,2,3).
\end{align}
By direct calculations, the governing equations of $(w,z)$ read that
\begin{align}\label{2.5}
\begin{split}
\pa_3w=\frac{\gamma-1}{16c_v\gamma}(w-z)^2S_r - \frac{m(\gamma-1)}{8r}(w^2-z^2), \\
\pa_1z=\frac{\gamma-1}{16c_v\gamma}(w-z)^2S_r + \frac{m(\gamma-1)}{8r}(w^2-z^2).
\end{split}
\end{align}
In order to deal with the term $S_r$, we consider the Lagrangian coordinate variable $\xi=\xi(r,t)$, which is defined by
\begin{align}\label{2.6}
\xi=\int_{0}^r\zeta^m\rho(\zeta,t)\ {\rm d}\zeta.
\end{align}
Denote
\begin{equation}
\label{st}
\widetilde{S}(\xi,t)=S(r,t)
\end{equation}
It follows by \eqref{2.6} and \eqref{a1} that
\begin{align}\label{2.7}
S_r=r^m\rho \widetilde{S}_\xi =\gamma_k r^m e^{-\fr{S}{c_\gamma}}h^{\fr{2}{\gamma-1}}\widetilde{S}_\xi,
\end{align}
and subsequently
\begin{align}\label{2.8}
S_{rr}=&\Big(\fr{m}{r}+\fr{\rho_r}{\rho}\Big)S_r+r^{2m}\rho^2 \widetilde{S}_{\xi\xi} \notag \\
=&\Big(\fr{m}{r}+\fr{\rho_r}{\rho}\Big)r^m\rho \widetilde{S}_\xi+r^{2m}\rho^2 \widetilde{S}_{\xi\xi} \notag \\
=& \Big(\fr{m}{r}+\fr{2}{\gamma-1}\fr{h_r}{h}-\fr{\gamma_k}{c_\gamma}r^{m}  e^{-\fr{S}{c_\gamma}}h^{\fr{2}{\gamma-1}}\widetilde{S}_\xi\Big)\gamma_k r^m e^{-\fr{S}{c_\gamma}}h^{\fr{2}{\gamma-1}}\widetilde{S}_\xi \notag \\
&+\gamma_{k}^2r^{2m}e^{-\fr{2S}{c_\gamma}}h^{\fr{4}{\gamma-1}} \widetilde{S}_{\xi\xi}.
\end{align}
Putting \eqref{2.7} into \eqref{2.5} yields
\begin{align}\label{2.9}
\begin{split}
\pa_3w=\frac{m(\gamma-1)}{8r} \bigg\{\frac{\gamma_k}{2m c_v\gamma}r^{m+1}e^{-\fr{S}{c_\gamma}}h^{\fr{2}{\gamma-1}}\widetilde{S}_\xi(w-z)^2 - (w^2-z^2)\bigg\}, \\
\pa_1z=\frac{m(\gamma-1)}{8r} \bigg\{\frac{\gamma_k}{2m c_v\gamma}r^{m+1}e^{-\fr{S}{c_\gamma}}h^{\fr{2}{\gamma-1}}\widetilde{S}_\xi(w-z)^2 + (w^2-z^2)\bigg\}.
\end{split}
\end{align}

We next calculate the gradient variables $(\alpha, \beta)$ defined in \eqref{1.4}. According to \eqref{2.1} and \eqref{2.4}, one has for $c_3\neq0$ and $r>0$
\begin{align}\label{2.10}
\alpha
&=-\frac{\partial_1(r^m\rho u)}{r^m\rho c_3}
= -\frac{\partial_t(r^m \rho u)}{r^m\rho c_3}-c_1\frac{\partial_r(r^m \rho u)}{r^m\rho c_3}  \notag
\\
&= \fr{1}{r^m\rho c_3}\bigg\{(r^m\rho u^2)_r +r^mp_r-(u-h)(r^m\rho u)_r\bigg\}   \notag \\
&=\fr{1}{\rho c_3}\bigg\{u^2\rho_r +2\rho uu_r +\fr{m}{r}\rho u^2 +h^2\rho_r +\fr{\rho}{\gamma c_v}h^2S_r-(u-h)\Big(\rho u_r+u\rho_r+\fr{m}{r}\rho u\Big)\bigg\}  \notag \\
&=\fr{1}{\rho c_3}\bigg\{\rho c_3u_r +hc_3\rho_r +\fr{\rho}{\gamma c_v}h^2S_r +\fr{m}{r}\rho uh\bigg\}   \notag \\
&=u_r +\fr{h}{\rho}\rho_r+\fr{h^2}{\gamma c_v c_3}S_r +\fr{m uh}{rc_3}   \notag
\\
&=u_r +\fr{h}{\rho}\cdot \fr{\rho}{\gamma-1}\Big(\fr{2}{h}h_r-\fr{1}{c_v}S_r\Big)+\fr{h^2}{\gamma c_v c_3}S_r +\fr{m uh}{rc_3}   \notag \\
&=u_r +\fr{2}{\gamma-1}h_r -\fr{h(\gamma u+h)}{\gamma c_\gamma c_3}S_r +\fr{m uh}{rc_3}.
\end{align}
Doing a similar computation for $\beta$ deduces for $c_1\neq0$ and $r>0$
\begin{align}\label{2.11}
\beta
&=-\frac{\partial_3(r^m\rho u)}{r^m\rho c_1}=
\fr{1}{r^m\rho c_1}\bigg\{(r^m\rho u^2)_r +r^mp_r-(u+h)(r^m\rho u)_r\bigg\}   \notag \\
&=\fr{1}{\rho c_1}\bigg\{\rho c_1u_r -hc_1\rho_r +\fr{\rho}{\gamma c_v}h^2S_r -\fr{m}{r}\rho uh\bigg\}   \notag \\
&=u_r -\fr{h}{\rho}\cdot \fr{\rho}{\gamma-1}\Big(\fr{2}{h}h_r-\fr{1}{c_v}S_r\Big)+\fr{h^2}{\gamma c_v c_1}S_r -\fr{m uh}{rc_1}   \notag \\
&=u_r -\fr{2}{\gamma-1}h_r +\fr{h(\gamma u-h)}{\gamma c_\gamma c_1}S_r-\fr{m uh}{rc_1}.
\end{align}
Combining \eqref{2.3}, \eqref{2.10} and \eqref{2.11}, we have
\begin{align}\label{2.12}
\begin{split}
\alpha=&w_r-\fr{h(\gamma u+h)}{\gamma c_\gamma c_3}S_r +\fr{m uh}{rc_3}, \\
\beta=&z_r+\fr{h(\gamma u-h)}{\gamma c_\gamma c_1}S_r-\fr{m uh}{rc_1},
\end{split}
\end{align}
provided that $c_1 c_3\neq0$ and $r>0$. Moreover, there holds
\begin{align}\label{2.13}
\begin{split}
u_r=&\fr{1}{2}(\alpha+\beta)-\fr{uh^2}{\gamma c_v c_1c_3}S_r +\fr{muh^2}{rc_1c_3}, \\
h_r=&\fr{\gamma-1}{4}\bigg(\alpha-\beta+\frac{2h(\gamma u^2-h^2)}{\gamma c_\gamma c_1c_3}S_r-\frac{2m u^2h}{rc_1c_3}\bigg).
\end{split}
\end{align}

For later applications, we here give the equations of $\pa_ih$
\begin{align}\label{a4}
\begin{split}
\partial_1h &= -\frac{\gamma-1}{2}h\alpha
    -\frac{h^2(\gamma u+h)}{2\gamma c_vc_3}r^m\rho \widetilde{S}_\xi
    -\frac{(\gamma-1)mu^2h}{2rc_3},
    \\
\partial_2h &= -\frac{\gamma-1}{4}h(\alpha+\beta)
    +\frac{(\gamma-1)uh^3}{2\gamma c_vc_1c_3}r^m\rho \widetilde{S}_\xi
    -\frac{(\gamma-1)mu^3h}{2rc_1c_3},
    \\
\partial_3h &= -\frac{\gamma-1}{2}h\beta
    +\frac{h^2(\gamma u-h)}{2\gamma c_vc_1}r^m\rho \widetilde{S}_\xi
    -\frac{(\gamma-1)mu^2h}{2rc_1},
\end{split}
\end{align}
which can be obtained through direct calculations by \eqref{2.1} and \eqref{2.13}.

\subsection{The assumptions and main results}\label{S22}

Let $T_0$ be any positive time and $b_1<b_2$ be any two positive numbers. We consider the initial data $(\rho_0(r), u_0(r), S_0(r))=(\rho(r,0), u(r,0), S(r,0))$ of \eqref{1.1} on the interval $[b_1, b_2]$ satisfying $\rho_0(r), u_0(r)\in C^1$, $S_0(r)\in C^2$
and $\rho_0(r)>0$. Denote
\begin{align}\label{2.14}
\begin{split}
h_0(r)=&\sqrt{K\gamma} e^{\frac{S_0(r)}{2c_v}}(\rho_0(r))^\frac{\gamma-1}{2},\\
w_0(r)=&u_0(r)+\frac{2}{\gamma-1}h_0(r),\quad z_0(r)=u_0(r)-\frac{2}{\gamma-1}h_0(r),
\end{split}
\end{align}
and
\begin{align}\label{2.15}
\tilde{w}=\max_{r\in[b_1,b_2]}w_0(r),\quad \tilde{z}=\min_{r\in[b_1,b_2]}z_0(r).
\end{align}
Then it suggests that
\begin{align}\label{2.16}
u_0(r)<\tilde{w},\quad h_0(r)<\fr{\gamma-1}{4}\tilde{w}=: h_d,\quad \forall\ r\in[b_1,b_2].
\end{align}
We further set
\begin{align}\label{2.18}
r_d=b_2+\tilde{w}T_0,
\end{align}
and denote
\begin{align}\label{2.17}
\begin{split}
\alpha_0(r)=&w_{0}'(r)-\fr{h_0(r)[\gamma u_0(r)+h_0(r)]}{\gamma c_\gamma [u_0(r)+h_0(r)]}S_{0}'(r) +\fr{m u_0(r)h_0(r)}{r[u_0(r)+h_0(r)]}, \\
\beta_0(r)=&z_{0}'(r)+\fr{h_0(r)[\gamma u_0(r)-h_0(r)]}{\gamma c_\gamma [u_0(r)-h_0(r)]}S_{0}'(r) -\fr{m u_0(r)h_0(r)}{r[u_0(r)-h_0(r)]},
\end{split}
\quad r\in[b_1,b_2].
\end{align}

We now state the assumptions of the paper.
\begin{asu}\label{asu1}
Let $1<\gamma<3$. For the initial data $(\rho_0(r), u_0(r), S_0(r) )$ of \eqref{1.1} on $[b_1, b_2]$, we assume that
\begin{align}\label{2.19}
0<\tilde{z}<\tilde{w}\leq \mathcal{C}_0,
\end{align}
for a positive finite constant $\mathcal{C}_0$.
\end{asu}

\begin{asu}\label{asu2}
For the initial data $(\rho_0(r), u_0(r), S_0(r) )$ of \eqref{1.1} on $[b_1, b_2]$, we assume that
\begin{align}\label{2.20}
|S_0(r)|\leq \mathcal{S}_0,\qquad 0\leq \fr{S_{0}'(r)}{r\rho_0(r)}\leq \mathcal{S}_1,
\end{align}
where $\mathcal{S}_0$ and $\mathcal{S}_1$ are two positive constants satisfying
\begin{align}\label{2.21}
\mathcal{S}_1\leq \fr{mc_v\gamma}{3\gamma_k r_{d}^{m+1}}e^{-\fr{\mathcal{S}_0}{c_\gamma}}h_{d}^{-\fr{2}{\gamma-1}}.
\end{align}
\end{asu}

\begin{asu}\label{asu3}
For the initial data $(\rho_0(r), u_0(r), S_0(r) )$ of \eqref{1.1} on $[b_1, b_2]$, we assume that
\begin{align}\label{2.22}
S_{0}''(r)-\Big(\fr{m}{r}+\fr{\rho_{0}'(r)}{\rho_0(r)}\Big)S_{0}'(r)\leq0.
\end{align}
\end{asu}

\begin{asu}\label{asu4}
For the initial data $(\rho_0(r), u_0(r), S_0(r) )$ of \eqref{1.1} on $[b_1, b_2]$, we assume that
\begin{align}\label{2.23}
\alpha_0(r)\geq0,\qquad \beta_0(r)\geq0.
\end{align}
\end{asu}

In addition, we denote
\begin{align}\label{2.24}
\begin{split}
\mathcal{S}_2=&\max_{r\in[b_1,b_2]} \bigg|\fr{S_{0}''(r)}{r^{2m}\rho_{0}^2(r)} -\fr{1}{r^{2m}\rho_{0}^2(r)}\Big(\fr{m}{r}+\fr{\rho_{0}'(r)}{\rho_0(r)}\Big)S_{0}'(r)\bigg|, \\
\mathcal{C}_1=&\max\bigg\{\max_{r\in[b_1,b_2]}|\alpha_0(r)|,\ \  \max_{r\in[b_1,b_2]}|\beta_0(r)|\bigg\},
\end{split}
\end{align}
which are two finite positive constants by the regularity assumptions of $(\rho_0(r), u_0(r), S_0(r))$.

We next give some remarks on the above assumptions.
\begin{rem}\label{r1}
Thanks to \eqref{2.7}, the inequalities \eqref{2.20} and \eqref{2.21} mean that
\begin{align}\label{2.25}
0\leq \gamma_k r_{d}^{m+1}e^{\fr{\mathcal{S}_0}{c_\gamma}}h_{d}^{\fr{2}{\gamma-1}}\widetilde{S}_\xi(\xi,0)\leq \fr{mc_v\gamma}{3},
\end{align}
for $\xi\in[\xi_1, \xi_2]$, where
$$
\xi_1=\int_{0}^{b_1}\zeta^m\rho_0(\zeta)\ {\rm d}\zeta,\quad \xi_2=\int_{0}^{b_2}\zeta^m\rho_0(\zeta)\ {\rm d}\zeta.
$$
This together with \eqref{2.19} will lead to an invariant domain of $(w,z)$, that is
$$
0<z(r,t)<w(r,t)< \tilde{w},
$$
for any $(r,t)\in\Omega_d$. Here and below $\Omega_d$ represents the characteristic triangle or quadrangle domain generated by interval $[b_1,b_2]$, see Figure \ref{fig1} for illustration. Particularly, it follows by the inequality $z>0$ that
$$
c_1=u-h>\fr{2}{\gamma-1}h-h=\fr{3-\gamma}{\gamma-1}h>0,
$$
for $1<\gamma<3$, which indicates that the supersonic expanding property of the solution can be preserved in the domain $\Omega_d$.
\end{rem}

\begin{figure}[h!]
\centering
\scalebox{0.45}{\includegraphics{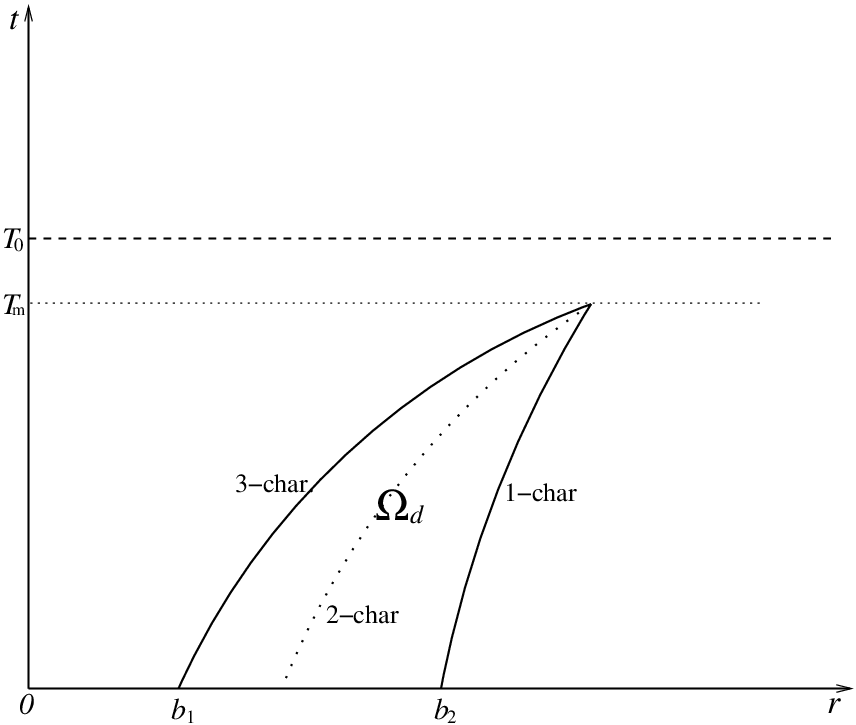}}
\quad \quad
\scalebox{0.45}{\includegraphics{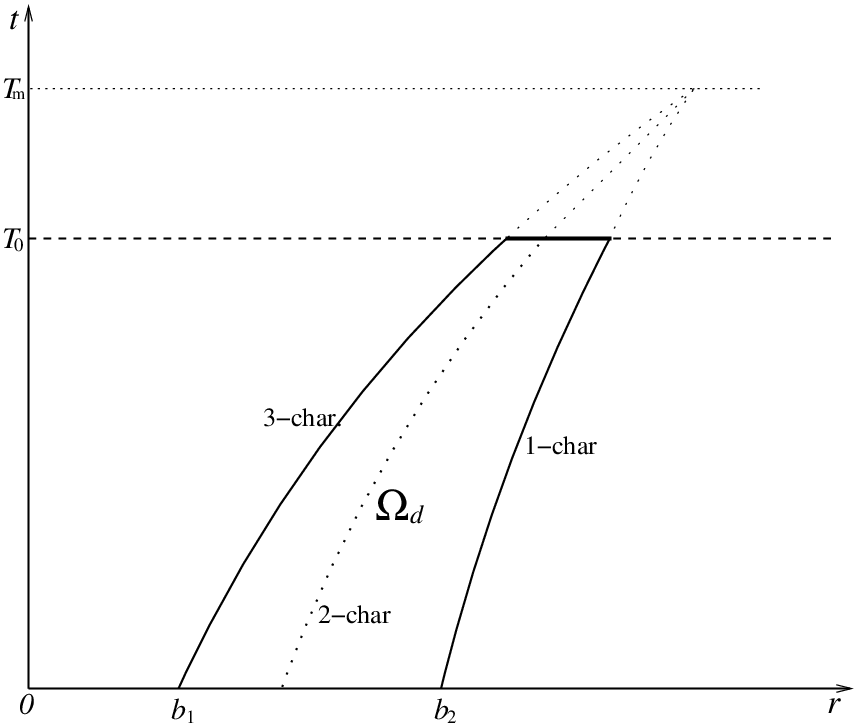}}
 \caption{The domain $\Omega_d$. When $T_0\geq T_m$ (Left), $\Omega_d$ is the characteristic triangle domain generated by interval $[b_1,b_2]$. When $T_0< T_m$ (Right), $\Omega_d$ is the characteristic quadrangle domain generated by interval $[b_1,b_2]$. Here $T_m$ is the intersection time of  the 3-characteristic originating from point $(b_1,0)$ and the 1-characteristic originating from  point $(b_2,0)$.
\label{fig1} }
\normalsize
\end{figure}

\begin{rem}\label{r2}
In view of \eqref{st}, \eqref{2.8} and \eqref{2.24}, the inequality \eqref{2.22} implies that on $[\xi_1,\xi_2]$
\begin{align}\label{2.26}
-\mathcal{S}_2\leq \widetilde{S}_{\xi\xi}(\xi,0)\leq0.
\end{align}
The initial properties of $\widetilde{S}_\xi$ and $\widetilde{S}_{\xi\xi}$ in \eqref{2.25} and \eqref{2.26} will be maintained in the domain $\Omega_d$ by the definition of $\xi$,
which is the main reason for introducing this Lagrangian coordinate variable.
\end{rem}
\begin{rem}\label{r3}
The inequality \eqref{2.23} will form an invariant domain of $(\alpha,\beta)$, that is
$$
\alpha(r,t)\geq0,\quad \beta(r,t)\geq0,\quad \forall\ (r,t)\in\Omega_d,
$$
which are used to establish the global existence of smooth solutions in $\Omega_d$.
\end{rem}

The main results of the paper can be stated as the following two theorems.
\begin{thm}\label{thm1}
Let Assumptions \ref{asu1}-\ref{asu4} be satisfied. Then, for the initial data $(\rho_0(r), u_0(r), S_0(r))$, the radially symmetric non-isentropic Euler equations \eqref{1.1} with \eqref{1.2} admits a smooth solution $(\rho, u, S)(r,t)$ on the entire domain $\Omega_d$. Furthermore, the solution fulfils
\begin{align}\label{2.27}
\begin{array}{c}
0<z(r,t)<w(r,t)<\tilde{w}, \quad |S(r,t)|\leq \mathcal{S}_0,\\
\dps \min_{\Omega_d}(\alpha,\beta)(r,t)\geq0, \quad \max_{\Omega_d}(\alpha,\beta)(r,t)< \widetilde{\mathcal{C}}_1,
\end{array}
\quad \forall\ (r,t)\in\Omega_d,
\end{align}
for some positive constant $\widetilde{\mathcal{C}}_1> \mathcal{C}_1$.
\end{thm}
\begin{rem}\label{r-b}
The main reason for the global existence result being confined to a finite domain $\Omega_d$ is that it relies heavily on the inequality \eqref{2.25}, see the proof process of Lemmas \ref{lem3} and \ref{lem4} below. The key application of this inequality is to derive a uniform estimate for the term $r^{m+1}\widetilde{S}_\xi$ (i.e. $rS_r$) independent of the spatial variable $r$, which is used to ensure the supersonic expanding property and the rarefactive property of the smooth solution.
\end{rem}

\begin{thm}\label{thm2}
Let Assumptions \ref{asu1} and \ref{asu2} hold. For any $T\in(0, T_0]$, there exists a constant $\mathcal{N}(T)$, depending on $T$, such that,  if
\begin{align}
\alpha_0(r^*)\leq -\mathcal{N}(T),\quad {\rm or}\quad  \beta_0(r^*)\leq -\mathcal{N}(T), \notag
\end{align}
for some number $r^*\in(b_1,b_2)$, then singularity forms before time $T$.
\end{thm}
\begin{rem}\label{r4}
The selection of the constant $\mathcal{N}(T)$ will be determined along with the analysis of the singularity in Section \ref{S52}, see \eqref{5.15} below.
\end{rem}

\section{The Riccati equations of $(\alpha,\beta)$}\label{S3}

In this section, we derive the Riccati-type equations for the gradient variables $(\alpha,\beta)$, which plays a crucial role in the analysis of the entire paper

\begin{lem}\label{lem1}
For smooth solution of \eqref{1.1}, \eqref{1.2}, the variables $(\alpha, \beta)$ defined in \eqref{2.12} satisfy the following Riccati-type equations when $r>0$ and $c_1c_2\neq 0$
\begin{align}\label{3.1}
\begin{split}
\pa_1 \beta =& -\frac{\gamma+1}{4} \beta^{2}
  -\frac{3-\gamma}{4} \alpha \beta-B_{1}\beta+ B_{2} \alpha +B_3, \\[4pt]
\pa_3 \alpha =& -\frac{\gamma+1}{4} \alpha^{2}
  -\frac{3-\gamma}{4} \alpha \beta-A_{1}\alpha+ A_{2} \beta +A_3,
\end{split}
\end{align}
where
\begin{align}\label{3.2}
\begin{split}
B_{1} &=\fr{mc_3}{2rc_{1}^2}\Big(\fr{\gamma-1}{2}u^2-h^2\Big) +\fr{(3-\gamma)mu^2h^2}{rc_{1}^2c_3} -\fr{h[3\gamma u^2+(\gamma+5)uh-h^2]}{4\gamma c_v c_1c_3}r^m\rho \widetilde{S}_\xi,   \\
B_{2} &=\fr{mc_3}{2rc_{1}^2}\Big(\fr{\gamma-1}{2}u^2-h^2\Big) -\fr{h[3\gamma u^2-(\gamma+3)uh-3h^2]}{4\gamma c_v c_{1}^2}r^m\rho \widetilde{S}_\xi,    \\
B_{3} &=\fr{u^3h^2}{\gamma c_{v}^2 r c_{1}^2 c_{3}}\Big(mc_v\gamma -r^{m+1}\rho \widetilde{S}_\xi\Big) r^m\rho \widetilde{S}_\xi -\fr{uh^2}{\gamma c_v c_1}r^{2m}\rho^2\widetilde{S}_{\xi\xi},
\end{split}
\end{align}
and
\begin{align}\label{3.3}
\begin{split}
A_{1} &=\fr{mc_1}{2rc_{3}^2}\Big(\fr{\gamma-1}{2}u^2-h^2\Big) +\fr{(3-\gamma)mu^2h^2}{rc_{1}c_{3}^2} +\fr{h[3\gamma u^2-(\gamma+5)uh-h^2]}{4\gamma c_v c_1c_3}r^m\rho \widetilde{S}_\xi,   \\
A_{2} &=\fr{mc_1}{2rc_{3}^2}\Big(\fr{\gamma-1}{2}u^2-h^2\Big) +\fr{h[3\gamma u^2+(\gamma+3)uh-3h^2]}{4\gamma c_v c_{3}^2}r^m\rho \widetilde{S}_\xi,    \\
A_{3} &=\fr{u^3h^2}{\gamma c_{v}^2 rc_{1}c_{3}^2}\Big(mc_v\gamma -r^{m+1}\rho \widetilde{S}_\xi\Big) r^m\rho \widetilde{S}_\xi -\fr{uh^2}{\gamma c_v c_3}r^{2m}\rho^2\widetilde{S}_{\xi\xi}.
\end{split}
\end{align}
with $\widetilde{S}$ defined in \eqref{st}.
\end{lem}
\begin{proof}
The derivation process is quite cumbersome, we will proceed step by step.

\subsection*{I. Derivation of the $\beta$-equation}
From the definition of $\beta$ in \eqref{2.11} and equation \eqref{2.2}, we obtain the expression for $\beta$
\begin{align}\label{3.4}
\partial_1\beta &= \partial_t\beta + c_1\partial_r \beta  =\left(u_r-\frac{2}{\gamma-1}h_r+\frac{h(\gamma u-h)}{\gamma c_\gamma c_1}S_r-\frac{m uh}{rc_1}\right)_t \nonumber\\
    &\quad +c_1\left(u_r-\frac{2}{\gamma-1}h_r+\frac{h(\gamma u-h)}{\gamma c_\gamma c_1}S_r-\frac{m uh}{rc_1}\right)_r
    \nonumber\\
    &=\underbrace{-\Big(\frac{2}{\gamma-1}h_{rt}+\frac{2}{\gamma-1}uh_{rr}+hu_{rr}\Big)}_{I_1}
    +\underbrace{\Big(u_{rt}+\frac{2}{\gamma-1}hh_{rr}+uu_{rr}\Big)}_{I_2}
    \nonumber\\
    &\quad +\underbrace{\frac{1}{\gamma c_\gamma}\Big(\frac{\gamma u-h}{c_1}h_tS_r+\frac{h}{c_1}(\gamma u_t-h_t)S_r-\frac{h(\gamma u-h)}{c_1^2}(u_t-h_t)S_r+\frac{h(\gamma u-h)}{c_1}S_{rt}\Big)}_{I_3}
    \nonumber\\
    &\quad \underbrace{-\frac{m}{r}\Big(\frac{h}{c_1}u_t+\frac{u}{c_1}h_t-\frac{uh}{c_1^2}(u_t-h_t)\Big)}_{I_4}
    \nonumber\\
    &\quad +\underbrace{\frac{1}{\gamma c_\gamma}\Big((\gamma u-h)h_rS_r+h(\gamma u_r-h_r)S_r-\frac{h(\gamma u-h)}{c_1}(u_r-h_r)S_r+h(\gamma u-h)S_{rr}\Big)}_{I_5}
    \nonumber\\
    &\quad \underbrace{-\frac{m}{r}\Big(hu_r+uh_r-\frac{uh}{c_1}(u_r-h_r)-\frac{uh}{r}\Big)}_{I_6}.
\end{align}

\noindent\textbf{Computation of $I_1$.}
Making use of \eqref{2.1} and \eqref{2.13}, we calculate
\begin{align}
I_1 &=\frac{\gamma+1}{\gamma-1}u_rh_r+\frac{m}{r}\left(hu_r+uh_r-\frac{uh}{r}\right)\nonumber\\
&=\frac{\gamma+1}{4}\left( \frac12(\alpha+\beta)-\frac{uh^2}{\gamma c_v c_1c_3}S_r +\frac{muh^2}{r c_1 c_3} \right) \left( \alpha-\beta+\frac{2h(\gamma u^2-h^2)}{\gamma c_\gamma c_1c_3}S_r-\frac{2m u^2h}{r c_1 c_3} \right) \nonumber\\
&\quad + \frac{mh}{2r}(\alpha+\beta)-\frac{muh^3}{\gamma c_v c_1c_3 r}S_r + \frac{m^2 u h^3}{r^2 c_1 c_3} \nonumber\\
&\quad+ \frac{m(\gamma-1)u}{4r} \left( \alpha-\beta+\frac{2h(\gamma u^2-h^2)}{\gamma c_\gamma c_1c_3}S_r-\frac{2m u^2h}{r c_1 c_3} \right) \nonumber\\
&=\frac{\gamma + 1}{8}\alpha^2 + \frac{(\gamma + 1)h(\gamma u + h)}{4\gamma c_{\gamma}  c_3}\alpha S_r + \frac{m\left[ (\gamma - 1)u^2 +2 h^2 \right]}{4r c_3}\alpha - \frac{\gamma + 1}{8}\beta^2
    \nonumber\\
&\quad+ \frac{(\gamma + 1)h(\gamma u - h)}{4\gamma c_{\gamma}  c_1}\beta S_r - \frac{m\left[ (\gamma - 1)u^2 +2 h^2 \right]}{4r c_1}\beta - \frac{(\gamma + 1)uh^3(\gamma u^2 - h^2)}{2 \gamma^2 c_v c_{\gamma} c_1^2 c_3^2}{S_r}^2 \nonumber\\
&\quad+ \frac{muh\left( (\gamma - 1)u^2 +2 h^2 \right)((\gamma - 2)h^2 + \gamma u^2)}{2 \gamma c_{\gamma} r c_1^2 c_3^2}S_r -\frac{m^2 uh \left( (\gamma - 1)u^4 + 2h^4 \right)}{2r^2 c_1^2 c_3^2} - \frac{muh}{r^2}. \notag
\end{align}

\noindent\textbf{Computation of $I_2$.}
In view of \eqref{2.1}, \eqref{2.8} and \eqref{2.13}, one obtains
\begin{align}
I_2 &=-\frac{2}{\gamma-1}h_r^2-u_r^2+\frac{2h}{c_v\gamma(\gamma-1)}h_rS_r
  +\frac{h^2}{c_v\gamma(\gamma-1)}S_{rr}\nonumber\\
  &= -\frac{\gamma-1}{8} \left( \alpha-\beta + \frac{2h(\gamma u^2-h^2)}{\gamma c_\gamma c_1c_3} S_r - \frac{2m u^2 h}{r c_1 c_3} \right)^2 \nonumber\\
  &\quad -\left(\frac12(\alpha+\beta) - \frac{uh^2}{\gamma c_v c_1c_3} S_r + \frac{m u h^2}{r c_1 c_3} \right)^2 \nonumber\\
  &\quad + \frac{h}{2c_v \gamma} \left( \alpha-\beta + \frac{2h(\gamma u^2-h^2)}{\gamma c_\gamma c_1c_3} S_r - \frac{2m u^2 h}{r c_1 c_3} \right) S_r \nonumber\\
  &\quad + \frac{h^2}{c_{\gamma} \gamma } \bigg( \frac{1}{2h}(\alpha-\beta) S_r - \frac{m h^2}{r c_1 c_3} S_r + \frac{(\gamma-1) h^2}{\gamma c_\gamma c_1 c_3} S_r^2 + \gamma_k^2 r^{2m} e^{-\frac{2S}{c_\gamma}} h^{\frac{4}{\gamma-1}} \widetilde{S}_{\xi\xi} \bigg)\nonumber\\
  &=  -\frac{\gamma + 1}{8}\alpha^2 - \frac{3 - \gamma}{4}\alpha\beta - \frac{h(\gamma u - h)((\gamma - 2)u - h)}{2c_{\gamma} \gamma c_1 c_3}\alpha S_r + \frac{muh((\gamma - 1)u - 2h)}{2r c_1 c_3}\alpha  \nonumber\\
  &\quad-\frac{\gamma + 1}{8}\beta^2 + \frac{h(\gamma u + h)((\gamma - 2)u + h)}{2c_{\gamma} \gamma c_1 c_3}\beta S_r - \frac{muh\left[ (\gamma - 1)u + 2h \right]}{2r c_1 c_3}\beta \nonumber\\
  &\quad - \frac{h^2(\gamma u^2 + h^2)((\gamma - 2)u^2 + h^2)}{2c_v c_{\gamma} \gamma^2 c_1^2 c_3^2}S_r^2 \nonumber\\
  &\quad+\frac{mh^2((\gamma - 1)u^2 + uh + h^2)((\gamma - 1)u^2 - uh + h^2)}{c_{\gamma} \gamma r c_1^2 c_3^2}S_r \nonumber\\
   &\quad +  \frac{ h^2}{\gamma c_{\gamma} } \gamma_k^2 r^{2m} e^{-\frac{2S}{c_\gamma}} h^{\frac{4}{\gamma-1}}  \widetilde{S}_{\xi\xi}-\frac{m^2 u^2 h^2 \left( (\gamma - 1)u^2 + 2h^2 \right)}{2r^2 c_1^2 c_3^2}. \notag
\end{align}

\noindent\textbf{Computation of $I_3$.}
We apply \eqref{2.1} and \eqref{2.8} again to acquire
\begin{align}
I_3 &= \frac{1}{\gamma c_\gamma}\bigg\{
       -\frac{h^2(\gamma-1)}{{c_1}^2}u_tS_r+\frac{\gamma u^2-2uh+h^2}{{c_1}^2}h_tS_r- \frac{h(\gamma u-h)}{c_1}u_rS_r-
       \frac{uh(\gamma u-h)}{c_1}S_{rr}\bigg\}\nonumber\\
       &=
-\frac{h^2(\gamma-1)}{\gamma c_\gamma c_1^2}\left( -\frac{c_3}{2}\alpha - \frac{c_1}{2}\beta \right)S_r \nonumber \\
&\quad +\frac{\gamma u^2-2uh+h^2}{\gamma c_\gamma c_1^2}\left( -\frac{(\gamma-1)c_3}{4}\alpha + \frac{(\gamma-1)c_1}{4}\beta - \frac{uh}{2c_v} S_r \right)S_r \notag \\
&\quad - \frac{h(\gamma u-h)}{\gamma c_\gamma c_1}\left( \frac12(\alpha+\beta) - \frac{uh^2}{\gamma c_v c_1 c_3} S_r + \frac{m u h^2}{r c_1 c_3} \right)S_r \nonumber \\
&\quad -\frac{uh(\gamma u-h)}{\gamma c_\gamma c_1}\bigg( \frac{1}{2h}(\alpha-\beta) S_r - \frac{m h^2}{r c_1 c_3} S_r + \frac{(\gamma-1) h^2}{\gamma c_\gamma c_1 c_3} S_r^2 \notag + \gamma_k^2 r^{2m} e^{-\frac{2S}{c_\gamma}} h^{\frac{4}{\gamma-1}} \widetilde{S}_{\xi\xi} \bigg)
\\
&= -\frac{c_3(\gamma(\gamma+1)u^2-4\gamma uh + (3-\gamma)h^2)}{4\gamma c_\gamma c_1^2}\alpha S_r
    + \frac{\gamma(\gamma+1)u^2-4\gamma uh+(3\gamma-1)h^2}{4\gamma c_\gamma c_1}\beta S_r
    \nonumber\\
    &-\frac{uh(\gamma u^2-2uh+h^2)}{2\gamma c_\gamma c_vc_1^2} S_r^2
    -\frac{uh(\gamma u-h)}{\gamma c_\gamma c_1} \gamma_k^2 r^{2m} e^{-\frac{2S}{c_\gamma}} h^{\frac{4}{\gamma-1}}  \widetilde{S}_{\xi\xi}. \notag
\end{align}

\noindent\textbf{Computation of $I_4$.}
It follows by \eqref{2.1} that
 \begin{align}
 I_4 &= -\frac{m}{r}\bigg(\frac{h}{c_1}u_t+\frac{u}{c_1}h_t-\frac{uh}{c_1^2}(u_t-h_t)\bigg)= \frac{m}{r c_1^2}\left( h^2 u_t - u^2 h_t \right) \nonumber\\
    &= \frac{m}{r c_1^2}\left\{ h^2\left( -\frac{c_3}{2}\alpha - \frac{c_1}{2}\beta \right) - u^2\left( -\frac{(\gamma-1)c_3}{4}\alpha + \frac{(\gamma-1)c_1}{4}\beta - \frac{u h}{2c_v} S_r \right) \right\} \nonumber\\
    &= -\frac{m c_3\big( 2h^2 - (\gamma-1)u^2 \big)}{4r c_1^2} \alpha - \frac{m\big( 2h^2 + (\gamma-1)u^2 \big)}{4r c_1} \beta + \frac{m u^3 h}{2c_v r c_1^2} S_r. \notag
 \end{align}

\noindent\textbf{Computation of $I_5$.}
Taking advantage of \eqref{2.8} and \eqref{2.13}, we have
\begin{align}
I_5
&=\frac{1}{\gamma c_\gamma}\left\{-\frac{h^2(\gamma-1)}{c_1}u_rS_r+\frac{\gamma u^2-2uh+h^2}{c_1}h_rS_r+h(\gamma u-h)S_{rr}\right\}\nonumber\\
&= -\frac{h^2(\gamma-1)}{\gamma c_\gamma c_1} \left( \frac{1}{2}(\alpha + \beta) - \frac{uh^2}{\gamma c_v c_1 c_3} S_r + \frac{muh^2}{r c_1 c_3} \right) S_r \nonumber\\
&\quad + \frac{(\gamma-1)(\gamma u^2 - 2uh + h^2)}{4\gamma c_\gamma c_1} \left( \alpha - \beta + \frac{2h(\gamma u^2 - h^2)}{\gamma c_v c_1 c_3} S_r - \frac{2mu^2 h}{r c_1 c_3} \right) S_r \nonumber\\
&\quad + \frac{h(\gamma u - h)}{\gamma c_\gamma} \left( \frac{1}{2h}(\alpha - \beta) S_r - \frac{mh^2}{r c_1 c_3} S_r + \frac{(\gamma-1)h^2}{\gamma c_v c_1 c_3} S_r^2 + \gamma k^2 r^{2m} e^{-\frac{2S}{c_\gamma}} h^{\frac{4}{\gamma-1}} \widetilde{S}_{\xi\xi} \right)\nonumber\\
&=  \frac{\gamma(\gamma+1)u^2-4\gamma uh+(3-\gamma)h^2}{4\gamma c_\gamma c_1}\alpha S_r
- \frac{\gamma(\gamma+1)u^2-4\gamma uh+(3\gamma-1)h^2}{4\gamma c_\gamma c_1}\beta  S_r \nonumber\\
& + \frac{h(\gamma u^2+h^2)(\gamma u^2-2uh+h^2)}{2\gamma^2 c_\gamma c_v c_1^2c_3} S_r^2 - \frac{mh(\gamma u^2-2uh+h^2)(\frac{\gamma-1}{2}u^2+h^2)}{\gamma c_\gamma rc_1^2c_3} S_r \nonumber\\
&\quad+ \frac{h(\gamma u - h)}{\gamma c_\gamma} \, \gamma_k^2 \, r^{2m} e^{-\frac{2S}{c_\gamma}} h^{\frac{4}{\gamma-1}} \widetilde{S}_{\xi\xi}. \notag
   \end{align}
\noindent\textbf{Computation of $I_6$.}
From \eqref{2.13}, we obtain
\begin{align}
 I_6
  &=\ \frac{m h^2}{rc_1}\left( \frac12(\alpha+\beta) - \frac{uh^2}{\gamma c_v c_1 c_3}S_r + \frac{m u h^2}{rc_1 c_3} \right) \nonumber\\
  &\quad-\frac{m u^2(\gamma-1)}{4rc_1}\left( \alpha-\beta + \frac{2h(\gamma u^2-h^2)}{\gamma c_\gamma c_1 c_3}S_r - \frac{2m u^2 h}{rc_1 c_3} \right) +\frac{m u h}{r^2}\nonumber\\
  &= \frac{m\big[2h^2 - (\gamma-1)u^2\big]}{4r c_1} \,\alpha + \frac{m\big[2h^2 + (\gamma-1)u^2\big]}{4r c_1} \,\beta \nonumber\\
  &\quad - \frac{muh(\gamma u^3-uh^2+2h^3)}{2\gamma r c_v{c_1}^2 c_3}  S_r + \frac{m u h}{r^2}\left( 1 + \frac{m(2h^3 + (\gamma-1)u^3)}{2 {c_1}^2 c_3}  \right). \notag
\end{align}

\noindent\textbf{Integration and simplification.} Putting the results of $I_i$ ($i=1,\cdots,6$) into \eqref{3.4} yields
\begin{align}\label{3.5}
\partial_1\beta=&I_1+I_2+I_3+I_4+I_5+I_6 \notag \\
=& -\frac{\gamma+1}{4}\beta^{2}
+ \frac{\gamma-3}{4}\alpha\beta +I_{7}\alpha S_{r}+I_{8}\alpha +I_{9}\beta S_{r} +I_{10}\beta \notag \\
&+I_{11} S_{r}^{2} +I_{12}S_r +I_{13} \widetilde{S}_{\xi\xi} +I_{14},
\end{align}
where
\begin{align*}
I_{7}=&\frac{(\gamma+1)h(\gamma u+h)}{4\gamma c_\gamma c_3}
   - \frac{h(\gamma u-h)((\gamma-2)u-h)}{2\gamma c_\gamma c_1 c_3} \\
&- \frac{c_3\big(\gamma(\gamma+1)u^{2}-4\gamma uh+(3-\gamma)h^{2}\big)}{4\gamma c_\gamma c_1^{2}} + \frac{\gamma(\gamma+1)u^{2}-4\gamma uh+(3-\gamma)h^{2}}{4\gamma c_\gamma c_1}, \\
I_{8}=&\frac{m\big((\gamma-1)u^{2}+2h^{2}\big)}{4r c_3}
   + \frac{muh\big((\gamma-1)u-2h\big)}{2r c_1 c_3} \\
&- \frac{mc_3\big(2h^{2}-(\gamma-1)u^{2}\big)}{4r c_1^{2}}  + \frac{m\big(2h^{2}-(\gamma-1)u^{2}\big)}{4r c_1}, \\
I_{9}=&\frac{(\gamma+1)h(\gamma u-h)}{4\gamma c_\gamma c_1}
   + \frac{h(\gamma u+h)((\gamma-2)u+h)}{2\gamma c_\gamma c_1 c_3}, \\
I_{10}=&- \frac{m\left((\gamma-1)u^2 + 2h^2\right)}{4r c_1}
-\frac{muh\left((\gamma-1)u + 2h\right)}{2r c_1 c_3},
\end{align*}
and
\begin{align*}
I_{11}=&-\frac{(\gamma+1)uh^{3}(\gamma u^{2}-h^{2})}{2\gamma^{2} c_v c_\gamma c_1^{2} c_3^{2}}
   - \frac{h^{2}(\gamma u^{2}+h^{2})\big((\gamma-2)u^{2}+h^{2}\big)}{2\gamma^{2} c_v c_\gamma c_1^{2} c_3^{2}} \\
& - \frac{uh(\gamma u^{2}-2uh+h^{2})}{2\gamma c_\gamma c_v c_1^{2}}
   + \frac{h(\gamma u^{2}+h^{2})(\gamma u^{2}-2uh+h^{2})}{2\gamma^{2} c_\gamma c_v c_1^{2} c_3}, \\
I_{12}=&\frac{muh\big((\gamma-1)u^{2}+2h^{2}\big)\big((\gamma-2)h^{2}+\gamma u^{2}\big)}{2\gamma c_\gamma r c_1^{2} c_3^{2}}  \\
&+ \frac{mh^{2}\big((\gamma-1)u^{2}+uh+h^{2}\big)\big((\gamma-1)u^{2}-uh+h^{2}\big)}{\gamma c_\gamma r c_1^{2} c_3^{2}} \\
&  + \frac{mu^{3}h}{2c_v r c_1^{2}}
   - \frac{mh(\gamma u^{2}-2uh+h^{2})\big(\frac{\gamma-1}{2}u^{2}+h^{2}\big)}{\gamma c_\gamma r c_1^{2} c_3}
   - \frac{muh(\gamma u^{3}-uh^{2}+2h^{3})}{2\gamma r c_v c_1^{2} c_3}, \\
I_{13}=&\frac{1}{\gamma c_{\gamma}}
\left( h^2 + h(\gamma u - h) - \frac{uh(\gamma u-h)}{c_1} \right)\gamma_k^2 r^{2m} e^{-\frac{2S}{c_\gamma}} h^{\frac{4}{\gamma-1}}, \\
I_{14}=&-\frac{m^{2}uh\big((\gamma-1)u^{4}+2h^{4}\big)}{2r^{2} c_1^{2} c_3^{2}}
   - \frac{m^{2}u^{2}h^{2}\big((\gamma-1)u^{2}+2h^{2}\big)}{2r^{2} c_1^{2} c_3^{2}} + \frac{m^{2}uh\big(2h^{3}+(\gamma-1)u^{3}\big)}{2r^{2} c_1^{2} c_3}.
\end{align*}

We next simplify the coefficients $I_{i}$ ($i=7,\cdots, 14$) one by one. For $I_{7}$ and $I_{8}$, we have
\begin{align*}
I_{7}
&= \frac{h}{4\gamma c_\gamma c_1 c_3} \Big\{ (\gamma+1)(\gamma u+h)(u-h) - 2(\gamma u-h)((\gamma-2)u-h) \Big\} \\
&\quad - \frac{\gamma(\gamma+1)u^{2}-4\gamma uh+(3-\gamma)h^{2}}{4\gamma c_\gamma c_1^2} ( c_3 - c_1  ) \\
&= \frac{h}{4\gamma c_\gamma c_1^2 c_3} \Big\{  (\gamma-1) \big( -3\gamma u^3 - (2\gamma-3)u^2 h + (\gamma+6)uh^2 + 3h^3 \big) \Big\} \\
&= \frac{h}{4\gamma c_\gamma c_1^2 c_3} \Big\{ (\gamma-1) \big( 3h^2 + (\gamma+3)u h -3\gamma u^2  \big)(u+h) \Big\} \\
&= -\frac{h[ 3\gamma u^2-(\gamma+3)u h-3h^2  ]}{4\gamma c_v c_1^2 },
\end{align*}
and
\begin{align*}
I_{8}
&= \frac{m\big((\gamma-1)u^{2}+2h^{2}\big)}{4r (u+h)}
   + \frac{m u h\big((\gamma-1)u-2h\big)}{2r (u-h)(u+h)}
  - \frac{m\big(2h^{2}-(\gamma-1)u^{2}\big)}{4r(u-h)^{2}}
   \cdot 2h \\
&=\frac{m\big((\gamma-1)u^{2}+2h^{2}\big)}{4r (u+h)}
   +\frac{mh((\gamma-1)u^3 - hu^2 - h^3)}{r(u-h)^2(u+h)}\\
&=\frac{m}{4r(u-h)^2(u+h)} \Big\{ (\gamma - 1)u^4 + 2(\gamma - 1)u^3 h + (\gamma - 1)u^2 h^2 - 2u^2 h^2 - 4u h^3 - 2h^4 \Big\} \\
&=\frac{m}{4r(u-h)^2(u+h)} \Big\{  (\gamma-1)u^2(u^2+2uh+h^2)-2h^2(u^2+2uh+h^2) \Big\} \\
&= \frac{mc_3}{2rc_1^{2}}\left(\frac{\gamma-1}{2}u^{2} - h^{2}\right).
\end{align*}
For $I_{9}$ and $I_{10}$, we deduce
\begin{align*}
I_{9}
&=\frac{h}{4\gamma c_\gamma c_1 c_3} \left\{
      (\gamma+1)(\gamma u-h)(u+h)
      + 2(\gamma u+h)\big((\gamma-2)u+h\big)
    \right\} \\
&= \frac{h}{4\gamma c_\gamma c_1 c_3} \left\{
      3\gamma(\gamma-1)u^2 + (\gamma+5)(\gamma-1)u h + (1-\gamma)h^2
    \right\} \\
&= \frac{h\big(3\gamma u^2 + (\gamma+5)u h - h^2\big)}{4\gamma c_v c_1 c_3},
\end{align*}
and
\begin{align*}
I_{10}
=&-\bigg(\frac{m\big((\gamma-1)u^{2}+2h^{2}\big)}{4r c_1}- \frac{m c_3\big((\gamma-1)u^{2}-2h^{2}\big)}{4r c_1^2}\bigg) \\
&
- \frac{m u h\big((\gamma-1)u+2h\big)}{2r c_1 c_3}
- \frac{m c_3\big((\gamma-1)u^{2}-2h^{2}\big)}{4r c_1^2} \\
=& \frac{muh ((\gamma-1)u-2h)}{2r c_1^2}- \frac{m u h\big((\gamma-1)u+2h\big)}{2r c_1 c_3}
- \frac{m c_3\big((\gamma-1)u^{2}-2h^{2}\big)}{4r c_1^2} \\
=& -\frac{(3-\gamma) m u^2 h^2}{r c_1^2 c_3}- \frac{m c_3}{r c_1^2}\left( \frac{\gamma-1}{2}u^{2} - h^{2} \right).
\end{align*}
For $I_{11}$ and $I_{12}$, one has
\begin{align*}
I_{11}
&= \frac{1}{2\gamma^{2} c_v c_\gamma c_1^{2} c_3^{2}} \bigg\{
   -(\gamma+1)uh^{3}(\gamma u^{2}-h^{2})
   - h^{2}(\gamma u^{2}+h^{2})\big((\gamma-2)u^{2}+h^{2}\big) \\
&\quad - \gamma c_3^{2} uh(\gamma u^{2}-2uh+h^{2})
   +c_3 h(\gamma u^{2}+h^{2})(\gamma u^{2}-2uh+h^{2}) \bigg\} \\
&= \frac{2\gamma(\gamma-1)u^3 h^2(h-u)}{2\gamma^{2} c_v c_\gamma c_1^{2} c_3^2} = -\frac{u^3 h^2}{\gamma c_v^2 c_1^{2} c_3},
\end{align*}
and
\begin{align*}
I_{12}
=& \frac{mh\big((\gamma-1)u^{2}+2h^{2}\big)}{2\gamma c_\gamma r c_1^{2}c_3^{2}} \left[ u\big((\gamma-2)h^{2}+\gamma u^{2}\big)
- (u+h)(\gamma u^{2}-2uh+h^{2}) \right] \\
&+ \frac{mh^{2}\big((\gamma-1)u^{2}+uh+h^{2}\big)\big((\gamma-1)u^{2}-uh+h^{2}\big)}{\gamma c_\gamma r c_1^{2} c_3^{2}} \\
&+ \frac{mh}{2\gamma rc_v c_1^{2} c_3} \left[ \gamma u^{3}
- u(\gamma u^{3}-uh^{2}+2h^{3}) \right]\\
=&\frac{mh^2\left( (\gamma-1)u^3 h + \gamma u^4 + u^2 h^2 + 2 u h^3 \right)}{2\gamma c_v r c_1^{2} c_3^{2}}
  +\frac{muh^2\left( \gamma u^2 + uh -2 h^2 \right)}{2r \gamma c_v c_1^2 c_3}\\
=&\frac{mu h^2 }{2r\gamma c_v c_1^2 c_3^2} \left( 2\gamma u^2 h + 2\gamma u^3  \right)
= \frac{mu^3 h^2 }{r c_v c_1^2 c_3}.
\end{align*}
We compute $I_{13}$ and $I_{14}$ to achieve
\begin{align*}
I_{13}
&= \frac{\gamma_k^2 r^{2m} e^{-\frac{2S}{c_\gamma}} h^{\frac{4}{\gamma-1}} }{\gamma c_{\gamma}}
\cdot \frac{u h^2(1 - \gamma)}{u-h}  = -\frac{u h^2}{\gamma c_v c_1}\gamma_k^2 r^{2m} e^{-\frac{2S}{c_\gamma}} h^{\frac{4}{\gamma-1}},
\end{align*}
and
\begin{align*}
I_{14}
&= -\frac{m^{2}uh\big[ (\gamma-1)u^{3}(u+h) + 2h^{3}(u+h) \big]}{2r^{2} c_1^{2} c_3^{2}}
   + \frac{m^{2}uh\big(2h^{3}+(\gamma-1)u^{3}\big)}{2r^{2} c_1^{2} c_3} \\
&= -\frac{m^{2}uh\big(2h^{3}+(\gamma-1)u^{3}\big)}{2r^{2} c_1^{2} c_3}
   + \frac{m^{2}uh\big(2h^{3}+(\gamma-1)u^{3}\big)}{2r^{2} c_1^{2} c_3}= 0.
\end{align*}

We insert the results of $I_i$ ($i=7,\cdots,14$) into \eqref{3.5} to obtain
\begin{align*}
\partial_1\beta
=& -\frac{\gamma+1}{4}\beta^{2}
+ \frac{\gamma-3}{4}\alpha\beta +(I_{9} S_{r} +I_{10})\beta +(I_{7} S_{r}+I_{8})\alpha \notag \\
&+[(I_{11} S_{r} +I_{12})S_r +I_{13} \widetilde{S}_{\xi\xi}] +I_{14} \\
=& -\frac{\gamma+1}{4}\beta^{2}
+ \frac{\gamma-3}{4}\alpha\beta \\
&-\bigg\{ \frac{(3-\gamma) m u^2 h^2}{r c_1^2 c_3}+ \frac{m c_3}{r c_1^2}\left( \frac{\gamma-1}{2}u^{2} - h^{2} \right)
 -\frac{h\big(3\gamma u^2 + (\gamma+5)u h - h^2\big)}{4\gamma c_v c_1 c_3} S_r\bigg\}\beta \\
&+\bigg\{ \frac{mc_3}{2r c_1^2} \left( \frac{\gamma-1}{2}u^{2} - h^{2} \right)
 -\frac{h\big( 3\gamma u^2-(\gamma+3)u h-3h^2  \big)}{4\gamma c_v c_1^2 } S_r \bigg\}\alpha \\
&+\bigg\{-\frac{u^3 h^2}{\gamma c_v^2 c_1^{2} c_3}S_r^2
+ \frac{mu^3 h^2 }{r c_v c_1^2 c_3} S_r
-\frac{u h^2}{\gamma c_v c_1}\gamma_k^2 r^{2m} e^{-\frac{2S}{c_\gamma}} h^{\frac{4}{\gamma-1}} \widetilde{S}_{\xi\xi}\bigg\},
\end{align*}
which together with \eqref{2.7} yields the desired equation for $\beta$ in \eqref{3.1}.

\subsection*{II. Derivation of the $\alpha$-equation}

From the definition of $\beta$ in \eqref{2.11} and equation \eqref{2.2}, we similarly obtain the expression for $\alpha$
\begin{align}\label{3.6}
\partial_3\alpha=& \partial_t\alpha + c_3\partial_r \alpha =\left(u_r+\frac{2}{\gamma-1}h_r-\frac{h(\gamma u+h)}{\gamma c_\gamma c_3}S_r+\frac{m uh}{rc_3}\right)_t \nonumber\\
    & +c_3\left(u_r+\frac{2}{\gamma-1}h_r-\frac{h(\gamma u+h)}{\gamma c_\gamma c_3}S_r+\frac{m uh}{rc_3}\right)_r
    \nonumber\\
=&\underbrace{\Big(\frac{2}{\gamma-1}h_{rt}+\frac{2}{\gamma-1}uh_{rr}+hu_{rr}\Big)}_{J_1}
    +\underbrace{\Big(u_{rt}+\frac{2}{\gamma-1}hh_{rr}+uu_{rr}\Big)}_{J_2}
    \nonumber\\
    & \underbrace{-\frac{1}{\gamma c_\gamma}\Big(\frac{\gamma u+h}{c_3}h_tS_r+\frac{h}{c_3}(\gamma u_t+h_t)S_r-\frac{h(\gamma u+h)}{c_3^2}(u_t+h_t)S_r+\frac{h(\gamma u+h)}{c_3}S_{rt}\Big)}_{J_3}
    \nonumber\\
    & +\underbrace{\frac{m}{r}\Big(\frac{h}{c_3}u_t+\frac{u}{c_3}h_t -\frac{uh}{c_3^2}(u_t+h_t)\Big)}_{J_4}
    \nonumber\\
    &\quad\underbrace{-\frac{1}{\gamma c_\gamma}\Big((\gamma u+h)h_rS_r+h(\gamma u_r+h_r)S_r-\frac{h(\gamma u+h)}{c_3}(u_r+h_r)S_r+(h(\gamma u+h))S_{rr}\Big)}_{J_5}
    \nonumber\\
    & +\underbrace{\frac{m}{r}\Big(hu_r+uh_r-\frac{uh}{c_3}(u_r+h_r)-\frac{uh}{r}\Big)}_{J_6}.
\end{align}
We first note that $J_1=-I_1$ and $J_2=I_2$. Next we compute $J_{i}$ ($i=3,4,5,6$).

\noindent\textbf{Computation of $J_3$.} Employing \eqref{2.1} and
\eqref{2.8} arrives at
 \begin{align}
J_3
=& -\frac{1}{\gamma c_\gamma}\bigg\{
       \frac{h^2(\gamma-1)}{{c_3}^2}u_tS_r+\frac{\gamma u^2+2uh+h^2}{{c_3}^2}h_tS_r- \frac{h(\gamma u+h)}{c_3}u_rS_r- \frac{uh(\gamma u+h)}{c_3}S_{rr}\bigg\}\nonumber\\
=&
-\frac{h^2(\gamma-1)}{\gamma c_\gamma c_3^2}\left( -\frac{c_3}{2}\alpha - \frac{c_1}{2}\beta \right)S_r \nonumber \\
&  -\frac{\gamma u^2+2uh+h^2}{\gamma c_\gamma c_3^2}\left( -\frac{(\gamma-1)c_3}{4}\alpha + \frac{(\gamma-1)c_1}{4}\beta - \frac{uh}{2c_v} S_r \right)S_r \nonumber \\
& + \frac{h(\gamma u+h)}{\gamma c_\gamma c_3}\left( \frac12(\alpha+\beta) - \frac{uh^2}{\gamma c_v c_1 c_3} S_r + \frac{m u h^2}{r c_1 c_3} \right)S_r \nonumber \\
& + \frac{uh(\gamma u+h)}{\gamma c_\gamma c_3}\bigg\{ \frac{1}{2h}(\alpha-\beta) S_r - \frac{m h^2}{r c_1 c_3} S_r + \frac{(\gamma-1) h^2}{\gamma c_\gamma c_1 c_3} S_r^2 + \gamma_k^2 r^{2m} e^{-\frac{2S}{c_\gamma}} h^{\frac{4}{\gamma-1}} \widetilde{S}_{\xi\xi} \bigg\}\nonumber\\
= & \frac{\gamma(\gamma+1)u^2+4\gamma uh + (3\gamma-1)h^2}{4\gamma c_\gamma c_3} \alpha S_r
    - \frac{c_1(\gamma(\gamma+1)u^2+4\gamma uh+(3-\gamma)h^2)}{4\gamma c_\gamma c_3^2} \beta S_r  \nonumber\\
& + \frac{uh(\gamma u^2+2uh+h^2)}{2\gamma c_\gamma c_v c_3^2}  S_r^2 +  \frac{uh(\gamma u+h)}{\gamma c_\gamma c_3} \gamma_k^2 r^{2m} e^{-\frac{2S}{c_\gamma}} h^{\frac{4}{\gamma-1}}  \widetilde{S}_{\xi\xi}. \notag
 \end{align}

\noindent\textbf{Computation of $J_4$.} We directly calculate by \eqref{2.1} that
 \begin{align}
 J_4
=& \frac{m}{r}\left( \frac{h}{c_3}\left( 1 - \frac{u}{c_3} \right) u_t + \frac{u}{c_3}\left( 1 - \frac{h}{c_3} \right) h_t \right) = \frac{m}{r c_3^2}\left( h^2 u_t + u^2 h_t \right) \nonumber\\
=& \frac{m}{r c_3^2}\left\{ h^2\left( -\frac{c_3}{2}\alpha - \frac{c_1}{2}\beta \right) + u^2\left( -\frac{(\gamma-1)c_3}{4}\alpha + \frac{(\gamma-1)c_1}{4}\beta - \frac{u h}{2c_v} S_r \right) \right\} \nonumber\\
=& -\frac{m\big( 2h^2 + (\gamma-1)u^2 \big)}{4r c_3} \alpha - \frac{m c_1\big( 2h^2 - (\gamma-1)u^2 \big)}{4r c_3^2} \beta - \frac{m u^3 h}{2c_v r c_3^2} S_r. \notag
 \end{align}

\noindent\textbf{Computation of $J_5$.}
It follows by \eqref{2.8} and \eqref{2.13} that
 \begin{align}
J_5
= &-\frac{1}{\gamma c_\gamma}\left\{\frac{h^2(\gamma-1)}{c_3}u_rS_r+\frac{\gamma u^2+2uh+h^2}{c_3}h_rS_r+h(\gamma u+h)S_{rr}\right\}\notag\\
=& -\frac{h^2(\gamma-1)}{\gamma c_\gamma c_3}\left(\frac12(\alpha+\beta)-\frac{uh^2}{\gamma c_v c_1c_3}S_r+\frac{muh^2}{rc_1c_3}\right)S_r \notag\\
& -\frac{(\gamma-1)(\gamma u^2+2uh+h^2)}{4\gamma c_\gamma c_3}\left(\alpha-\beta+\frac{2h(\gamma u^2-h^2)}{\gamma c_\gamma c_1c_3}S_r-\frac{2mu^2h}{rc_1c_3}\right)S_r \notag\\
& -\frac{h(\gamma u+h)}{\gamma c_\gamma}\left(\frac{1}{2h}(\alpha-\beta)S_r-\frac{mh^2}{rc_1c_3}S_r+\frac{(\gamma-1)h^2}{\gamma c_\gamma c_1c_3}S_r^2+\gamma_k^2r^{2m}e^{-\frac{2S}{c_\gamma}}h^{\frac{4}{\gamma-1}} \widetilde{S}_{\xi\xi}\right)\notag\\
& -\frac{h(\gamma u+h)}{\gamma c_\gamma}\,\gamma_k^2 r^{2m}e^{-\frac{2S}{c_\gamma}}h^{\frac{4}{\gamma-1}}\widetilde{S}_{\xi\xi}\notag\\
=&- \frac{\gamma(\gamma+1)u^2+4\gamma uh + (3\gamma-1)h^2}{4\gamma c_\gamma c_3}\alpha S_r
+ \frac{\gamma(\gamma+1)u^2+4\gamma uh + (3-\gamma)h^2}{4\gamma c_\gamma c_3} \beta S_r\notag\\
&  -\frac{h(\gamma u^2 + h^2)(\gamma u^2+2uh+h^2)}{2\gamma^2 c_\gamma c_v c_1c_3^2} S_r^2
+\frac{mh(\gamma u^2+2uh+h^2)(\frac{\gamma-1}{2}u^2+h^2)}{\gamma c_\gamma rc_1c_3^2} S_r\notag\\
& -\frac{h(\gamma u+h)}{\gamma c_\gamma}\,\gamma_k^2 r^{2m}e^{-\frac{2S}{c_\gamma}}h^{\frac{4}{\gamma-1}}\widetilde{S}_{\xi\xi}. \notag
 \end{align}
\noindent\textbf{Computation of $J_6$.}
We use \eqref{2.13} to gain
 \begin{align}
J_6
=& \frac{m h^2}{rc_3}\left( \frac12(\alpha+\beta) - \frac{uh^2}{\gamma c_v c_1 c_3}S_r + \frac{m u h^2}{rc_1 c_3} \right) \nonumber\\
  & +\frac{m u^2(\gamma-1)}{4rc_3}\left( \alpha-\beta + \frac{2h(\gamma u^2-h^2)}{\gamma c_\gamma c_1 c_3}S_r - \frac{2m u^2 h}{rc_1 c_3} \right) -\frac{m u h}{r^2}\nonumber\\
=& \frac{m\big[2h^2 + (\gamma-1)u^2\big]}{4r c_3} \,\alpha + \frac{m\big[2h^2 - (\gamma-1)u^2\big]}{4r c_3} \,\beta \nonumber\\
  &  + \frac{muh(\gamma u^3-uh^2-2h^3)}{2\gamma r c_vc_1 {c_3}^2}  S_r - \frac{m u h}{r^2}\left( 1 - \frac{m(2h^3 - (\gamma-1)u^3)}{2 c_1 {c_3}^2} \right). \notag
 \end{align}

\noindent\textbf{Integration and simplification.} We substitute the results of $J_{i}$ ($i=1,\cdots,6$) into \eqref{3.6} to obtain
\begin{align}\label{3.7}
\partial_3 \alpha=&-I_1+I_2+J_3+J_{4}+J_{5}+J_{6}\notag\\
=&-\frac{\gamma+1}{4}\alpha^{2}
- \frac{3-\gamma}{4}\alpha\beta +J_7\alpha S_{r} +J_8\alpha +J_9\beta S_r +J_{10}\beta \notag \\
&+J_{11}S_{r}^2 +J_{12}S_r +J_{13}\widetilde{S}_{\xi\xi}  +J_{14},
\end{align}
where
\begin{align*}
J_7=&- \frac{(\gamma+1)h(\gamma u+h)}{4\gamma c_\gamma c_3}
   - \frac{h(\gamma u-h)((\gamma-2)u-h)}{2\gamma c_\gamma c_1 c_3}, \\
J_8=&- \frac{m\big((\gamma-1)u^2+2h^2\big)}{4r c_3}
   + \frac{muh\big((\gamma-1)u-2h\big)}{2r c_1 c_3}, \\
J_9=& -\frac{(\gamma+1)h(\gamma u-h)}{4\gamma c_\gamma c_1}
   + \frac{h(\gamma u+h)((\gamma-2)u+h)}{2\gamma c_\gamma c_1 c_3} \\
&- \frac{c_1\big(\gamma(\gamma+1)u^2+4\gamma uh+(3-\gamma)h^2\big)}{4\gamma c_\gamma c_3^2} + \frac{\gamma(\gamma+1)u^2+4\gamma uh+(3-\gamma)h^2}{4\gamma c_\gamma c_3}, \\
J_{10}=& \frac{m\big((\gamma-1)u^2+2h^2\big)}{4r c_1}
   - \frac{muh\big((\gamma-1)u+2h\big)}{2r c_1 c_3} \\
&- \frac{m c_1\big(2h^2-(\gamma-1)u^2\big)}{4r c_3^2} + \frac{m\big(2h^2-(\gamma-1)u^2\big)}{4r c_3},
\end{align*}
and
\begin{align*}
J_{11}=& \frac{(\gamma+1)uh^{3}(\gamma u^{2}-h^{2})}{2\gamma^{2} c_v c_\gamma c_1^{2} c_3^{2}}
   - \frac{h^{2}(\gamma u^{2}+h^{2})\big((\gamma-2)u^{2}+h^{2}\big)}{2\gamma^{2} c_v c_\gamma c_1^{2} c_3^{2}}  \\
&+ \frac{uh(\gamma u^{2}+2uh+h^{2})}{2\gamma c_\gamma c_v c_3^{2}}
   - \frac{h(\gamma u^{2}+h^{2})(\gamma u^{2}+2uh+h^{2})}{2\gamma^{2} c_\gamma c_v c_1 c_3^{2}}, \\
J_{12}=& -\frac{muh\big((\gamma-1)u^{2}+2h^{2}\big)\big((\gamma-2)h^{2}+\gamma u^{2}\big)}{2\gamma c_\gamma r c_1^{2} c_3^{2}} \\
&+ \frac{mh^{2}\big((\gamma-1)u^{2}+uh+h^{2}\big)\big((\gamma-1)u^{2}-uh+h^{2}\big)}{\gamma c_\gamma r c_1^{2} c_3^{2}}   \\
&- \frac{m u^{3} h}{2c_v r c_3^{2}}
   + \frac{mh(\gamma u^{2}+2uh+h^{2})\big(\frac{\gamma-1}{2}u^{2}+h^{2}\big)}{\gamma c_\gamma r c_1 c_3^{2}}
   + \frac{muh(\gamma u^{3}-uh^{2}-2h^{3})}{2\gamma r c_v c_1 c_3^{2}}, \\
J_{13}=& \frac{1}{\gamma c_{\gamma}}
\left( h^2 - h(\gamma u+h) + \frac{uh(\gamma u+h)}{c_3} \right)\gamma_k^2 r^{2m} e^{-\frac{2S}{c_\gamma}} h^{\frac{4}{\gamma-1}}, \\
J_{14}=&\frac{m^{2}uh\big((\gamma-1)u^{4}+2h^{4}\big)}{2r^{2} c_1^{2} c_3^{2}} - \frac{m^{2}u^{2}h^{2}\big((\gamma-1)u^{2}+2h^{2}\big)}{2r^{2} c_1^{2} c_3^{2}} + \frac{m^2u h\big(2h^{3}-(\gamma-1)u^{3}\big)}{2r^{2} c_1 c_3^{2}}.
\end{align*}

For $J_7$ and $J_8$, one deduces
\begin{align*}
J_7
&= -\frac{h}{4\gamma c_\gamma c_1 c_3} \bigg\{
      (\gamma+1)(\gamma u+h)(u-h)
      + 2(\gamma u-h)\big((\gamma-2)u-h\big)
    \bigg\} \\
&= -\frac{h}{4\gamma c_\gamma c_1 c_3} \bigg\{
      3\gamma(\gamma-1)u^2 - (\gamma+5)(\gamma-1)u h + (1-\gamma)h^2
    \bigg\} \\
&= -\frac{h\big(3\gamma u^2 - (\gamma+5)u h - h^2\big)}{4\gamma c_v c_1 c_3},
\end{align*}
and
\begin{align*}
J_8
=&\bigg(-\frac{m\big((\gamma-1)u^2+2h^2\big)}{4r c_3}
+ \frac{m c_1\big((\gamma-1)u^{2}-2h^{2}\big)}{4r c_3^2}\bigg) \\
&+  \frac{muh\big((\gamma-1)u-2h\big)}{2r c_1 c_3}- \frac{m c_1\big((\gamma-1)u^{2}-2h^{2}\big)}{4r c_3^2}  \\
=& -\frac{m u h \big( (\gamma-1) u + 2h \big)}{2r c_3^2} + \frac{muh\big((\gamma-1)u-2h\big)}{2r c_1 c_3}
- \frac{m c_1\big((\gamma-1)u^{2}-2h^{2}\big)}{4r c_3^2}  \\
=& -\frac{(3-\gamma) m u^2 h^2}{r c_1 c_3^2}- \frac{m c_1}{r c_3^2}\left( \frac{\gamma-1}{2}u^{2} - h^{2} \right).
\end{align*}
For $J_9$ and $J_{10}$, we have
\begin{align*}
J_9
=& \frac{h}{4\gamma c_\gamma c_1 c_3}\Big(-(\gamma+1)(\gamma u-h)(u+h) + 2(\gamma u+h)((\gamma-2)u+h)\Big)\\
  & + \frac{\big[\gamma(\gamma+1)u^2+4\gamma uh+(3-\gamma)h^2\big](c_3-c_1)}{4\gamma c_\gamma c_3^2} \\
=& \frac{h}{4\gamma c_\gamma c_1 c_3^2}\Big( 3\gamma(\gamma-1) u^3-(2\gamma-3)(\gamma-1)u^2 h-(\gamma+6)(\gamma-1)u h^2+3(\gamma-1)h^3 \Big) \\
=& \frac{h\big(3\gamma u^2+(\gamma+3)u h-3h^2\big)}{4\gamma c_v c_3^2 },
\end{align*}
and
\begin{align*}
J_{10}
&= \frac{m\big((\gamma-1)u^{2}+2h^{2}\big)}{4r (u-h)}
   - \frac{m u h\big((\gamma-1)u+2h\big)}{2r (u-h)(u+h)}
   + \frac{m (c_3-c_1) \big(2h^{2}-(\gamma-1)u^{2}\big)}{4r (u+h)^2} \\
&=\frac{m\big((\gamma-1)u^{2}+2h^{2}\big)}{4r (u-h)} -\frac{m h\big(  (\gamma-1)u^3 + hu^2 + h^3 \big)}{r (u-h)(u+h)^2}\\
&= \frac{m}{4r(u-h) (u+h)^2} \Big(
      (\gamma-1)u^{2}(u-h)^2 - 2h^{2}(u-h)^2 \Big)
= \frac{mc_1}{2r c_3^2} \left( \frac{\gamma-1}{2}u^{2} - h^{2} \right).
\end{align*}
For $J_{11}$ and $J_{12}$, we acquire
\begin{align*}
J_{11}
=& \frac{1}{2\gamma^{2} c_v c_\gamma c_1^{2} c_3^{2}} \bigg\{
   (\gamma+1)uh^{3}(\gamma u^{2}-h^{2}) - h^{2}(\gamma u^{2}+h^{2})\big((\gamma-2)u^{2}+h^{2}\big) \\
& + \gamma c_1^{2} uh (\gamma u^{2}+2uh+h^{2}) - c_1 h (\gamma u^{2}+h^{2})(\gamma u^{2}+2uh+h^{2}) \bigg\} \\
=& \frac{1}{2\gamma^{2} c_v c_\gamma c_1^{2} c_3^{2}} \Big(
   2\gamma(\gamma-1))u^3 h^3 -2\gamma(\gamma-1)u^4 h^2 \Big) \\
=& \frac{2\gamma(\gamma-1)u^3 h^2(h - u)}{2\gamma^{2} c_v c_\gamma c_1^{2} c_3^{2}} = -\frac{u^3 h^2}{\gamma c_v^2 c_1 c_3^2},
\end{align*}
and
\begin{align*}
J_{12}
=& -\frac{mh\big((\gamma-1)u^{2}+2h^{2}\big)}{2\gamma c_\gamma r c_1^{2} c_3^{2}} \Big( u\big((\gamma-2)h^{2}+\gamma u^{2}\big)
- c_1(\gamma u^{2}+2uh+h^{2}) \Big) \\
& + \frac{mh^{2}\big((\gamma-1)u^{2}+uh+h^{2}\big)\big((\gamma-1)u^{2}-uh+h^{2}\big)}{\gamma c_\gamma r c_1^{2} c_3^{2}} \\
& - \frac{muh}{2\gamma r c_v c_1 c_3^{2}} \Big( \gamma u^{2}(u-h) - (\gamma u^{3}-uh^{2}-2h^{3}) \Big)\\
=&\frac{mh^2\left( -u^3 h + \gamma u^4 + u^2 h^2 - 2 u h^3 \right)}{2r\gamma c_v  c_1^{2} c_3^{2}}
+\frac{muh^2\left( \gamma u^2 - uh -2 h^2 \right)}{2r \gamma c_v c_1 c_3^2}\\
=&\frac{mu h^2 }{2r\gamma c_v c_1^2 c_3^2} ( 2\gamma u^3 - 2\gamma u^2 h )
= \frac{mu^3 h^2 }{r c_v c_1 c_3^2}.
\end{align*}
Finally, we calculate $J_{13}$ and $J_{14}$ to attain
\begin{align*}
J_{13}=& \frac{\gamma_k^2 r^{2m} e^{-\frac{2S}{c_\gamma}} h^{\frac{4}{\gamma-1}}}{\gamma c_{\gamma}}
\cdot \frac{u h^2 (1-\gamma)}{u+h} = -\frac{u h^2}{\gamma c_v c_3}\gamma_k^2 r^{2m} e^{-\frac{2S}{c_\gamma}} h^{\frac{4}{\gamma-1}},
\end{align*}
and
\begin{align*}
J_{14}
=& \frac{m^{2}uh\big[ (\gamma-1)u^{3}(u-h) - 2h^{3}(u-h) \big]}{2r^{2}c_1^{2} c_3^{2}}
   + \frac{m^{2}uh\big(2h^{3}-(\gamma-1)u^{3}\big)}{2r^{2} c_1 c_3^{2}} \\
=& \frac{m^{2}uh\,\big((\gamma-1)u^{3}-2h^{3}\big)}{2r^{2}c_1 c_3^{2}}
   - \frac{m^{2}uh\,\big((\gamma-1)u^{3}-2h^{3}\big)}{2r^{2}c_1 c_3^{2}}= 0.
\end{align*}

Summing up the results of $J_{i}$ ($i=7,\cdots,14$), we obtain by \eqref{3.7}
\begin{align*}
\partial_3 \alpha =&-\frac{\gamma+1}{4}\alpha^{2}
- \frac{3-\gamma}{4}\alpha\beta +(J_7 S_{r} +J_8)\alpha +(J_9 S_r +J_{10})\beta \\
&+[(J_{11}S_{r}  +J_{12})S_r +J_{13}\widetilde{S}_{\xi\xi}]  +J_{14} \\
=&-\frac{\gamma+1}{4}\alpha^{2}
- \frac{3-\gamma}{4}\alpha\beta \\
&-\bigg\{\frac{(3-\gamma) m u^2 h^2}{r c_1 c_3^2}+ \frac{m c_1}{r c_3^2}\left( \frac{\gamma-1}{2}u^{2} - h^{2} \right)
 + \frac{h\big(3\gamma u^2 - (\gamma+5)u h - h^2\big)}{4\gamma c_v c_1 c_3} S_r\bigg\}\alpha \\
& +\bigg\{\frac{mc_1}{2r c_3^2} \left( \frac{\gamma-1}{2}u^{2} - h^{2} \right)
+ \frac{h\big[3\gamma u^2+(\gamma+3)u h-3h^2\big]}{4\gamma c_v c_3^2 } S_r\bigg\}\beta \\ &+\bigg\{-\frac{u^3 h^2}{\gamma c_v^2 c_1 c_3^2}S_r^2
+ \frac{mu^3 h^2}{r c_v c_1 c_3^2} S_r
- \frac{u h^2}{\gamma c_v c_3}\gamma_k^2 r^{2m} e^{-\frac{2S}{c_\gamma}} h^{\frac{4}{\gamma-1}} \widetilde{S}_{\xi\xi}\bigg\},
\end{align*}
from which, we apply \eqref{2.7} to achieve the equation for $\alpha$ in \eqref{3.1}. The proof of the lemma is complete.
\end{proof}

\section{$C^1$-estimates of the solution}\label{S4}

In this section, we derive the a priori $C^1$-estimates of the solution in the domain $\Omega_{d}$ by establishing invariant domains of $(w,z)$ and $(\alpha,\beta)$.

\begin{lem}\label{lem2}
The quantities $\widetilde{S}_\xi$ and $\widetilde{S}_{\xi\xi}$ are invariant over 2-characters, that is
\begin{align}\label{4.1}
\pa_2\widetilde{S}_\xi=0,\qquad \pa_2\widetilde{S}_{\xi\xi}=0.
\end{align}
\end{lem}
\begin{proof}
According to the last equation of \eqref{2.1} and
the definition of the Lagrangian coordinate variable $\xi$, we know that
$$
\widetilde{S}_t=\pa_2S=0,
$$
from which one has
\begin{align*}
\pa_2\widetilde{S}_\xi=(\widetilde{S}_\xi)_t=(\widetilde{S}_t)_\xi=0,
\end{align*}
and
\begin{align*}
\pa_2\widetilde{S}_{\xi\xi}=(\widetilde{S}_{\xi\xi})_t=(\widetilde{S}_t)_{\xi\xi}=0.
\end{align*}
The proof of the lemma is complete.
\end{proof}
\begin{rem}\label{r5}
In view of Lemma \ref{lem2} and Remarks \ref{r1} and \ref{r2}, we can directly obtain, in the domain $\Omega_d$
\begin{align}\label{4.2}
|S(r,t)|\leq \mathcal{S}_0,\quad 0\leq \gamma_k r_{d}^{m+1}e^{\fr{\mathcal{S}_0}{c_\gamma}}h_{d}^{\fr{2}{\gamma-1}}\widetilde{S}_\xi\leq \fr{mc_v\gamma}{3},
\end{align}
if Assumption \ref{asu2} holds, and
\begin{align}\label{4.3}
-\mathcal{S}_2\leq \widetilde{S}_{\xi\xi}\leq0,
\end{align}
if Assumption \ref{asu3} holds.
\end{rem}

\begin{lem}\label{lem3}
Let Assumptions \ref{asu1} and \ref{asu2} hold. Then any smooth solution $(\rho, u, S)$ of \eqref{1.1} with \eqref{1.2} in the domain $\Omega_d$ satisfies
\begin{align}\label{4.4}
0<z(r,t)\leq w(r,t)<\tilde{w}.
\end{align}
\end{lem}
\begin{rem}\label{r-a}
The inequality $z(r,t)\leq w(r,t)$ will be improved to $z(r,t)< w(r,t)$ by establishing the positive lower bound of $h$ later.
\end{rem}
\begin{proof}
The inequality $z\leq w$ follows from their definitions in \eqref{2.3}. We next show the inequality $z>0$ in $\Omega_d$.
Thanks to Assumption \ref{asu1}, we see that $0<z(r,0)\leq w(r,0)< \tilde{w}$ on $[b_1,b_2]$.
Assume that there exists a point $(r',t')\in\Omega_d$ such that $z(r',t')=0$ and $z(r,t)>0$ in $\Omega_d\cap\{t<t'\}$. This means that $t'$ is the first time such that $z=0$. We draw the $1$-characteristic $r=r_1(t)$ through the point $(r',t')$ up to the line $t=0$. Then, by the fact $z(r_1(0),0)>0$, there exists at least a point $(r_1(t''),t'')$ with $t''\leq t'$ on the curve $r=r_1(t)$ such that
\begin{align}\label{4.5}
\pa_1z(r_1(t''),t'')<0.
\end{align}
On the other hand, it follows by \eqref{2.9} and the fact $\widetilde{S}_\xi\geq0$ in \eqref{4.2} that
\begin{align*}
\pa_1z\geq\frac{m(\gamma-1)}{8r} (w^2-z^2),
\end{align*}
which implies that
\begin{align}\label{4.6}
\pa_1z(r_1(t''),t'')\geq\frac{m(\gamma-1)}{8r_1(t'')} (w^2-z^2)(r_1(t''),t'')\geq0,
\end{align}
which contradicts \eqref{4.5}. Thus we obtain $z>0$ in $\Omega_d$.

We next show $w< \tilde{w}$ in $\Omega_d$. Assume that there exists a first time point $(r',t')\in\Omega_d$ such that $w(r',t')=\tilde{w}$. Then we draw the $3$-characteristic $r=r_3(t)$ through the point $(r',t')$ up to the line $t=0$. We know by the fact $w(r,0)<\tilde{w}$ that there exists a point $(r_3(t''),t'')$ with $t''\leq t'$ on the curve $r=r_3(t)$ such that
\begin{align}\label{4.7}
\pa_3w(r_3(t''),t'')>0.
\end{align}
On the other hand, one has
\begin{align*}
\fr{{\rm d}r_3(t)}{{\rm d}t}=& u(r_3(t),t)+h(r_3(t),t) \\
\leq& u(r_3(t),t) +\fr{2}{\gamma-1}h(r_3(t),t)= w(r_3(t),t)\leq\tilde{w},
\end{align*}
for $t\leq t''$, from which and \eqref{2.18} we acquire
\begin{align}\label{4.8}
r''\leq r_3(0)+\tilde{w}t''\leq b_2+\tilde{w}T_0=r_d.
\end{align}
Combining \eqref{4.2} and \eqref{4.8} yields
\begin{align}\label{4.9}
&\frac{\gamma_k}{2m c_v\gamma}r^{m+1}e^{-\fr{S}{c_\gamma}}h^{\fr{2}{\gamma-1}}\widetilde{S}_\xi\Big|_{(r'',t'')} \notag \\
\leq &\fr{1}{6}\bigg(\fr{r''}{r_d}\bigg)^{m+1}\bigg(\fr{h(r'',t'')}{h_d}\bigg)^{\fr{2}{\gamma-1}} e^{-\fr{S(r'',t'')+\mathcal{S}_0}{c_\gamma}}\leq \fr{1}{6}.
\end{align}
Here we used the fact $h(r'',t'')=\fr{\gamma-1}{4}(w-z)(r'',t'')<\fr{\gamma-1}{4}w(r'',t'')= h_d$.
Inserting \eqref{4.9} into the first equation of \eqref{2.9} gains
\begin{align}\label{4.10}
\pa_3w(r'',t'')\leq \frac{m(\gamma-1)}{8r''} \bigg\{\fr{1}{6}(w-z)^2 - (w^2-z^2)\bigg\}(r'',t'')\leq0,
\end{align}
a contradiction with \eqref{4.7}. This completes the proof of the lemma.
\end{proof}
\begin{rem}\label{r6}
From Lemma \ref{lem3}, we conclude that in the domain $\Omega_d$
\begin{equation}\label{4.11}
h<h_d,\qquad u<\tilde{w}, \qquad r<r_d,
\end{equation}
and
\begin{equation}\label{4.12}
z<w,\qquad c_1 > \frac{3-\gamma}{\gamma-1} h>0,\quad \fr{h}{c_1}< \fr{\gamma-1}{3-\gamma},\quad \fr{u}{c_1}=1+\fr{h}{c_1}<\fr{2}{3-\gamma},
\end{equation}
provided that $h>0$. Moreover, one has by \eqref{4.2}, \eqref{4.11} and \eqref{a1}
\begin{equation}\label{a2}
0\leq r^m\rho \widetilde{S}_\xi\leq \fr{mc_v\gamma}{3r_d}.
\end{equation}
\end{rem}
\begin{lem}\label{lem4}
Let Assumptions \ref{asu1}-\ref{asu3} hold. For the smooth solution $(\rho, u, S)$ of \eqref{1.1}, \eqref{1.2} with $\rho>0$, the terms $A_i$ and $B_i$ (i=1,2,3) defined in \eqref{3.2} and \eqref{3.3} satisfy
\begin{equation}\label{4.13}
A_i > 0, \qquad  B_i > 0,\qquad B_1-B_2>0,\qquad A_1-A_2>0.
\end{equation}
In addition, there hold
\begin{equation}\label{a3}
A_3,B_3\leq \fr{\gamma m^2\tilde{w}^2(\gamma-1)^2}{3b_1r_d(3-\gamma)^2} +\fr{\gamma_{k}^2(\gamma-1)\tilde{w}}{c_v\gamma(3-\gamma)}r_{d}^{2m}h_{d}^{\fr{\gamma+3}{\gamma-1}} e^{\fr{2\mathcal{S}_0}{c_\gamma}}\mathcal{S}_2=:L.
\end{equation}
\end{lem}
\begin{proof}
By means of \eqref{4.11} and \eqref{a2}, we first obtain
\begin{align}
mc_v\gamma -r^{m+1}\rho \widetilde{S}_\xi
\geq mc_v\gamma-r\fr{mc_v\gamma}{3r_d}\geq mc_v\gamma-\fr{mc_v\gamma}{3}=\fr{2mc_v\gamma}{3}>0, \notag
\end{align}
from which and \eqref{4.3} we see by the expressions of $A_3$ and $B_3$ that
$$
A_3>0,\qquad B_3>0.
$$

Note by the fact $z>0$ that
$$
\fr{\gamma-1}{2}u^2-h^2>0,\qquad [3\gamma u^2+(\gamma+3)uh-3h^2]>0,
$$
which together with \eqref{a2} and the expression of $A_2$ yields $A_2>0$. We now discuss the term $B_2$. It follows by \eqref{3.2}, \eqref{4.12} and \eqref{a2} that
\begin{align}\label{4.15}
B_{2} \geq & \fr{mc_3}{2rc_{1}^2}\Big(\fr{\gamma-1}{2}u^2-h^2\Big) -\fr{h[3\gamma u^2-(\gamma+3)uh-3h^2]}{4 c_{1}^2}\cdot \fr{m}{3r_d} \notag \\
\geq & \fr{m}{12r_dc_{1}^2}\bigg\{6(u+h)\Big(\fr{\gamma-1}{2}u^2-h^2\Big) -h[3\gamma u^2-(\gamma+3)uh-3h^2]\bigg\} \notag \\
\geq & \fr{mh}{12r_dc_{1}^2}\bigg\{6\fr{\gamma+1}{\gamma-1}\Big(\fr{\gamma-1}{2}u^2-h^2\Big) - [3\gamma u^2-(\gamma+3)uh-3h^2]\bigg\} \notag \\
= & \fr{mh}{12r_dc_{1}^2}\bigg\{3u^2-\fr{3(\gamma+3)}{\gamma-1}h^2 +(\gamma+3)uh\bigg\} \notag \\
\geq & \fr{mh^3}{12r_dc_{1}^2}\bigg\{\fr{12}{(\gamma-1)^2}-\fr{3(\gamma+3)}{\gamma-1}  +(\gamma+3)\fr{2}{\gamma-1}\bigg\} \notag \\
= & \fr{mh^3}{12r_dc_{1}^2}\cdot\fr{12-(\gamma+3)(\gamma-1)}{(\gamma-1)^2}>0,
\end{align}
by $1<\gamma<3$.

Next we show the inequality $B_1-B_2>0$, which also leads $B_1>0$ by \eqref{4.15}. According to the expressions of $B_1$ and $B_2$ in \eqref{3.2}, we directly calculate by \eqref{a2} and \eqref{4.4}
\begin{align}\label{4.16}
B_1-B_{2}=&\fr{3-\gamma}{r}\fr{mu^2h^2}{c_{1}^2c_3}+\fr{hr^m\rho \widetilde{S}_\xi}{4c_v\gamma c_1}\bigg(\fr{3\gamma u^2-(\gamma+3)uh-3h^2}{c_1} -\fr{3\gamma u^2+(\gamma+5)uh-h^2}{c_3}\bigg) \notag \\
=&\fr{3-\gamma}{r}\fr{mu^2h^2}{c_{1}^2c_3}+\fr{hr^m\rho \widetilde{S}_\xi}{4c_v\gamma c_1}\cdot\fr{4h[(\gamma-3)u^2+(u^2-h^2)]}{c_1c_3} \notag \\
>&\fr{3-\gamma}{r}\fr{mu^2h^2}{c_{1}^2c_3}-\fr{hm}{12 r_d c_1}\cdot\fr{4h(3-\gamma)u^2}{c_1c_3}> \fr{2(3-\gamma)mu^2h^2}{3r_dc_{1}^2c_3}>0.
\end{align}
Similarly, one has
\begin{align}\label{4.17}
A_1-A_{2}
=&\fr{3-\gamma}{r}\fr{mu^2h^2}{c_{1}c_{3}^2}+\fr{hr^m\rho \widetilde{S}_\xi}{4c_v\gamma c_3}\cdot\fr{4h[(\gamma-3)u^2+(u^2-h^2)]}{c_1c_3} \notag \\
>&  \fr{2(3-\gamma)mu^2h^2}{3r_dc_{1}c_{3}^2}>0.
\end{align}

Finally, we derive the upper bounds of $A_3$ and $B_3$. In the light of their expressions in \eqref{3.2} and \eqref{3.3}, we find by \eqref{4.11}-\eqref{a2} and \eqref{4.3} that
\begin{align}\label{a5}
A_3,B_{3}\leq & \fr{u^3h^2}{\gamma c_{v}^2 r c_{1}^2 c_{3}}mc_v\gamma \cdot r^m\rho \widetilde{S}_\xi +\fr{uh^2}{\gamma c_v c_1}r^{2m}\rho^2|\widetilde{S}_{\xi\xi}| \notag \\
\leq & \fr{mu^2}{c_{v} r} \Big(\fr{u}{c_3}\Big)\Big(\fr{h^2}{c_{1}^2}\Big)\cdot\fr{mc_v\gamma}{3r_d} +\fr{uh}{\gamma c_v }\Big(\fr{h}{c_1}\Big)r^{2m}\gamma_{k}^2h^{\fr{4}{\gamma-1}}e^{-\fr{2S}{c_\gamma}}\mathcal{S}_2  \notag \\
\leq & \fr{m\tilde{w}^2}{c_{v} b_1}\Big(\fr{\gamma-1}{3-\gamma}\Big)^2\cdot\fr{mc_v\gamma}{3r_d} +\fr{\tilde{w}h_d}{\gamma c_v }\Big(\fr{\gamma-1}{3-\gamma}\Big)r_{d}^{2m}\gamma_{k}^2h_{d}^{\fr{4}{\gamma-1}} e^{\fr{2\mathcal{S}_0}{c_\gamma}}\mathcal{S}_2=L.
\end{align}
The proof of the lemma is complete.
\end{proof}

We next establish the invariant domains of $(\alpha,\beta)$.
\begin{lem}\label{lem5}
Let the conditions of Lemma \ref{lem4} hold. We further suppose that Assumption \ref{asu4} be satisfied. Then, in the domain $\Omega_d$, the smooth solution $(\rho, u, S)$ of \eqref{1.1} with \eqref{1.2} fulfils
\begin{align}\label{4.18}
\alpha(r,t)\geq0,\quad \beta(r,t)\geq0,\quad \alpha(r,t)< \widetilde{\mathcal{C}}_1,\quad \beta(r,t)< \widetilde{\mathcal{C}}_1,
\end{align}
where
$$
\widetilde{\mathcal{C}}_1=\max\{\mathcal{C}_1+1, \ \sqrt{2L}\}.
$$
Here the constant $\mathcal{C}_1$ and $L$ are given in \eqref{2.24} and \eqref{a3}, respectively.
\end{lem}
\begin{proof}
We still apply the contradiction argument to show the lemma. Suppose there exists a point $(r',t')$ in $\Omega_d$ such that $\alpha(r', t')<0$. We draw the 3-characteristic $r=r_3(t)$ through the point $(r',t')$ up to the line $t=0$. Since $\alpha(r_3(0),0)\geq0$, there exists a point $(r_3(t''),t'')$ such that $\alpha(r_3(t''),t'')=0$ and
\begin{align}\label{4.19}
\pa_3\alpha(r_3(t''),t'')\leq0.
\end{align}
If $\beta(r_3(t''),t'')\geq0$, we utilize \eqref{3.1} and \eqref{4.13} to achieve
\begin{align*}
\pa_3 \alpha(r_3(t''),t'') =& \bigg\{ -\dfrac{\gamma+1}{4} \alpha^{2}
  -\dfrac{3-\gamma}{4} \alpha \beta-A_{1}\alpha+ A_{2} \beta +A_3\bigg\}(r_3(t''),t'') \\
  =& (A_{2} \beta +A_3)(r_3(t''),t'')>0,
\end{align*}
which contradicts \eqref{4.19}. If $\beta(r_3(t''),t'')<0$, we then draw the 1-characteristic $r=r_1(t)$ through the point $(r'',t'')$ up to the line $t=0$. Hence there exists a point $(r_3(t'''),t''')$ such that $\beta(r_1(t'''),t''')=0$, $\alpha(r_1(t'''),t''')\geq0$ and
\begin{align}\label{4.20}
\pa_1\beta(r_1(t'''),t''')\leq0.
\end{align}
Making use of \eqref{3.1} and \eqref{4.13} again gives
\begin{align*}
\pa_1 \beta(r_1(t'''),t''') =& \bigg\{-\dfrac{\gamma+1}{4} \beta^{2}
  -\dfrac{3-\gamma}{4} \alpha \beta-B_{1}\beta+ B_{2} \alpha +B_3\bigg\}(r_1(t'''),t''') \\
  =&(B_{2} \alpha +B_3)(r_1(t'''),t''')>0,
\end{align*}
a contradiction with \eqref{4.20}. Thus we have verified that $\alpha(r,t)\geq0$ and $\beta(r,t)\geq0$ in the domain $\Omega_d$.

To show the upper bound inequalities of $(\alpha,\beta)$, we first rewrite \eqref{3.1} as the following form
\begin{align}\label{4.21}
\begin{split}
\pa_1 \beta =& -\dfrac{\gamma+1}{4} \beta^{2}
  -\dfrac{3-\gamma}{4} \alpha \beta-(B_{1}-B_2)\beta- B_{2} (\beta-\alpha) +B_3, \\[8pt]
\pa_3 \alpha =& -\dfrac{\gamma+1}{4} \alpha^{2}
  -\dfrac{3-\gamma}{4} \alpha \beta-(A_{1}-A_2)\alpha- A_{2} (\alpha-\beta) +A_3.
\end{split}
\end{align}
Without loss of generality, we suppose that there exists a first time point $(r',t')$ in $\Omega_d$ such that $\beta(r',t')=\widetilde{\mathcal{C}}_1$ and $\beta(r,t)<\widetilde{\mathcal{C}}_1$, $\alpha(r,t)\leq \widetilde{\mathcal{C}}_1$ for $(r,t)\in\Omega_d\cap\{t<t'\}$. Then one sees that
\begin{align}\label{4.22}
\pa_1\beta(r',t')\geq0.
\end{align}
On the other hand, we find by \eqref{4.21}, \eqref{4.13} and \eqref{a3} that
\begin{align*}
&\pa_1 \beta(r',t') = \bigg\{-\dfrac{\gamma+1}{4} \beta^{2}
  -\dfrac{3-\gamma}{4} \alpha \beta-(B_{1}-B_2)\beta- B_{2} (\beta-\alpha) +B_3\bigg\}(r',t') \\
  =&-\dfrac{\gamma+1}{4}\widetilde{\mathcal{C}}_{1}^2 -\dfrac{3-\gamma}{4}\widetilde{\mathcal{C}}_{1}\alpha(r',t') -(B_{1}-B_2)\widetilde{\mathcal{C}}_{1} -B_{2} [\widetilde{\mathcal{C}}_{1}-\alpha(r',t')] +B_3(r',t') \\
  \leq & -\dfrac{\gamma+1}{4}\widetilde{\mathcal{C}}_{1}^2 +L\leq\fr{1-\gamma}{2}L<0,
\end{align*}
which contradicts \eqref{4.22}. Here we used the previous conclusion $\alpha\geq0$.
The proof of the lemma is finished.
\end{proof}

Note that the proof of the upper bounds of $(\alpha, \beta)$ in Lemma \ref{lem5} relies on the non-negativity of $\alpha(r,0)$ and $\beta(r,0)$, which is invalid in the proof of singularity formation later, as the initial data at that time is negative at some point. Therefore, we need to re-establish the upper bounds of $(\alpha, \beta)$ in the new scene. We first derive the equations for the weighted variables of $(\alpha, \beta)$.

\begin{lem}\label{lem6}
Let $\lambda>0$ be any real number. For smooth solution of \eqref{1.1} with \eqref{1.2}, the
weighted variables $h^{-\lambda}\alpha$ and $h^{-\lambda}\beta$ satisfy the following equations
\begin{align}\label{4.23}
\begin{split}
& \pa_1 (h^{-\lambda}\beta) =-\dfrac{\gamma+1}{4}h^{\lambda} (h^{-\lambda}\beta)^{2}
  +\Big(\fr{\gamma-1}{2}\lambda-\dfrac{3-\gamma}{4}\Big) h^{\lambda}(h^{-\lambda}\alpha) (h^{-\lambda}\beta) \\
  &\ +\Big(-B_{1}+\fr{h(\gamma u+h)}{2\gamma c_v c_3}r^m\rho \widetilde{S}_\xi \lambda +\fr{(\gamma-1)mu^2}{2rc_3}\lambda\Big)(h^{-\lambda}\beta)+ B_{2} (h^{-\lambda}\alpha) +h^{-\lambda}B_3, \\[4pt]
& \pa_3 (h^{-\lambda}\alpha) =-\dfrac{\gamma+1}{4}h^{\lambda} (h^{-\lambda}\alpha)^{2}
  +\Big(\fr{\gamma-1}{2}\lambda-\dfrac{3-\gamma}{4}\Big) h^{\lambda}(h^{-\lambda}\alpha) (h^{-\lambda}\beta)\\
  &\ +\Big(-A_{1}-\fr{h(\gamma u-h)}{2\gamma c_v c_1}r^m\rho \widetilde{S}_\xi \lambda +\fr{(\gamma-1)mu^2}{2rc_1}\lambda\Big)(h^{-\lambda}\alpha)+ A_{2} (h^{-\lambda}\beta) +h^{-\lambda}A_3.
\end{split}
\end{align}
\end{lem}
\begin{proof}
We here only show the equation of $\alpha$, and the equation of $\beta$ can be obtained similarly. By \eqref{3.1} and \eqref{a4}, one directly calculates
\begin{align*}
&\pa_3 (h^{-\lambda}\alpha) =h^{-\lambda}\pa_3\alpha -\lambda h^{-\lambda-1}\alpha\pa_3h \\
=&h^{-\lambda}\bigg\{-\dfrac{\gamma+1}{4} \alpha^{2}
  -\frac{3-\gamma}{4} \alpha \beta-A_{1}\alpha+ A_{2} \beta +A_3\bigg\} \\
&\ -\lambda h^{-\lambda-1}\alpha\bigg\{-\frac{\gamma-1}{2}h\beta
    +\frac{h^2(\gamma u-h)}{2\gamma c_vc_1}r^m\rho \widetilde{S}_\xi
    -\frac{(\gamma-1)mu^2h}{2rc_1}\bigg\} \\
=&-\dfrac{\gamma+1}{4}h^{-\lambda} \alpha^{2} +\Big(\frac{\gamma-1}{2}\lambda -\frac{3-\gamma}{4}\Big)h^{-\lambda}\alpha \beta +A_2(h^{-\lambda}\beta) \\
&\ +\Big(-A_{1}-\fr{h(\gamma u-h)}{2\gamma c_v c_1}r^m\rho \widetilde{S}_\xi \lambda +\fr{(\gamma-1)mu^2}{2rc_1}\lambda\Big)(h^{-\lambda}\alpha) +h^{-\lambda}A_3.
\end{align*}
Putting the identities $\alpha=h^\lambda(h^{-\lambda}\alpha)$ and $\beta=h^\lambda(h^{-\lambda}\beta)$ into the above yields the desired equation of $h^{-\lambda}\alpha$. The proof of the lemma is completed.
\end{proof}

We have the weighted upper bounds of $(\alpha, \beta)$ without assuming that $\alpha(r,0)$ and $\beta(r,0)$ are non-negative.
\begin{lem}\label{lem7}
Let Assumptions \ref{asu1} and \ref{asu2} hold. The smooth solution $(\rho, u, S)$ of \eqref{1.1} with \eqref{1.2} satisfies
\begin{align}\label{4.24}
h^{-\fr{2}{\gamma-1}}\alpha(r,t)< \widetilde{\mathcal{C}}_2,\quad h^{-\fr{2}{\gamma-1}}\beta(r,t)< \widetilde{\mathcal{C}}_2,
\end{align}
in the domain $\Omega_d$, where
\begin{align}\label{4.25}
\widetilde{\mathcal{C}}_2=e^{MT_0} \max\bigg\{\fr{\gamma_k\mathcal{C}_1}{\dps\min_{r\in[b_1,b_2]}\rho_0(r)} e^{\fr{\mathcal{S}_0}{c_\gamma}},\ \widetilde{L}\bigg\},
\end{align}
and
\begin{align}\label{a6}
\begin{split}
M=&\fr{\gamma mh_d}{3(\gamma-1)r_d}+\fr{2m\tilde{w}}{(3-\gamma)b_1} +1, \\[4pt]
\widetilde{L}=&\fr{m^2\gamma \tilde{w}^2}{3r_{d}^2h_{d}^{\fr{2}{\gamma-1}}} \Big(\fr{\gamma-1}{3-\gamma}\Big)^2 +\fr{\gamma_{k}^2\tilde{w}}{\gamma c_v }\Big(\fr{\gamma-1}{3-\gamma}\Big)r_{d}^{2m}h_{d}^{\fr{\gamma+1}{\gamma-1}} e^{\fr{2\mathcal{S}_0}{c_\gamma}}\mathcal{S}_2.
\end{split}
\end{align}
\end{lem}
\begin{proof}
We take $\lambda=\fr{2}{\gamma-1}$ in \eqref{4.23} such that $\fr{\gamma-1}{2}\lambda-\fr{3-\gamma}{4}=\fr{\gamma+1}{4}$ and denote
\begin{align}\label{4.26}
\hat{\alpha}=e^{-Mt}h^{-\fr{2}{\gamma-1}}\alpha,\qquad \hat{\beta}=e^{-Mt}h^{-\fr{2}{\gamma-1}}\beta.
\end{align}
Then the governing equations of $(\hat\alpha,\hat\beta)$ read that
\begin{align}\label{4.27}
\begin{split}
\pa_1 \hat{\beta} =& \Big(\fr{\gamma+1}{4}h^{\fr{2}{\gamma-1}}e^{Mt}\hat{\beta} +B_2\Big)(\hat\alpha-\hat\beta) -(B_1-B_2)\hat\beta \\
&\ +\Big(\fr{h(\gamma u+h)}{\gamma c_\gamma c_3}r^m\rho \widetilde{S}_\xi +\fr{mu^2}{rc_3}-M\Big)\hat\beta +h^{-\fr{2}{\gamma-1}}e^{-Mt}B_3, \\
\pa_3 \hat{\alpha} =& \Big(\fr{\gamma+1}{4}h^{\fr{2}{\gamma-1}}e^{Mt}\hat{\alpha} +A_2\Big)(\hat\beta-\hat\alpha) -(A_1-A_2)\hat\beta \\
&\ +\Big(-\fr{h(\gamma u-h)}{\gamma c_\gamma c_1}r^m\rho \widetilde{S}_\xi +\fr{mu^2}{rc_1}-M\Big)\hat\beta +h^{-\fr{2}{\gamma-1}}e^{-Mt}A_3.
\end{split}
\end{align}
Recalling \eqref{4.11}-\eqref{a2}, one finds by the definition of $M_1$ that
\begin{align}\label{4.28}
\fr{h(\gamma u+h)}{\gamma c_\gamma c_3}r^m\rho \widetilde{S}_\xi +\fr{mu^2}{rc_3} \leq & \fr{\gamma h_d}{\gamma c_\gamma }\cdot\fr{m\gamma c_v}{3r_d} +\fr{m\tilde{w}}{b_1} \notag \\
=&\fr{\gamma mh_d}{3(\gamma-1)r_d}+\fr{m\tilde{w}}{b_1}\leq M-1,
\end{align}
and
\begin{align}\label{4.29}
& -\fr{h(\gamma u-h)}{\gamma c_\gamma c_1}r^m\rho \widetilde{S}_\xi +\fr{mu^2}{rc_1} \notag \\ \leq& \fr{mu^2}{rc_1}
\leq \fr{m\tilde{w}}{b_1}\cdot \Big(1+\fr{h}{c_1}\Big)\leq \fr{2m\tilde{w}}{(3-\gamma)b_1}\leq M-1.
\end{align}
Furthermore, similar to \eqref{a5}, we can employ the expressions of $A_3$, $B_3$ in \eqref{3.2} and \eqref{3.3} to estimate
\begin{align}\label{4.30}
h^{-\fr{2}{\gamma-1}}A_3, h^{-\fr{2}{\gamma-1}}B_3\leq & h^{-\fr{2}{\gamma-1}}\bigg\{\fr{u^3h^2}{\gamma c_{v}^2 r c_{1}^2 c_{3}}mc_v\gamma \cdot r^m\rho \widetilde{S}_\xi +\fr{uh^2}{\gamma c_v c_1}r^{2m}\rho^2|\widetilde{S}_{\xi\xi}|\bigg\} \notag \\
\leq & \fr{mu^2}{c_{v} r} \Big(\fr{u}{c_3}\Big)\Big(\fr{h^2}{c_{1}^2}\Big)\cdot r^m \gamma_k e^{-\fr{S}{c_\gamma}}\cdot \fr{mc_v\gamma}{3}\Big(\gamma_k r_{d}^{m+1}e^{\fr{\mathcal{S}_0}{c_\gamma}}h_{d}^{\fr{2}{\gamma-1}}\Big)^{-1} \notag \\
&\ +\fr{uh}{\gamma c_v} \Big(\fr{h}{c_1}\Big)r^{2m}\gamma_{k}^2h^{\fr{2}{\gamma-1}} e^{-\fr{2S}{c_\gamma}}\mathcal{S}_2 \notag \\
\leq & \fr{m^2\gamma \tilde{w}^2}{3r_{d}^2h_{d}^{\fr{2}{\gamma-1}}} \Big(\fr{\gamma-1}{3-\gamma}\Big)^2 +\fr{\gamma_{k}^2\tilde{w}}{\gamma c_v }\Big(\fr{\gamma-1}{3-\gamma}\Big)r_{d}^{2m}h_{d}^{\fr{\gamma+1}{\gamma-1}} e^{\fr{2\mathcal{S}_0}{c_\gamma}}\mathcal{S}_2 =\widetilde{L}.
\end{align}

We next claim that $(\hat\alpha, \hat\beta)$ satisfy the following estimates
\begin{align}\label{4.31}
\hat\alpha(r,t), \hat\beta(r,t)<\max\bigg\{\fr{\gamma_k\mathcal{C}_1}{\dps\min_{r\in[b_1,b_2]}\rho_0(r)} e^{\fr{\mathcal{S}_0}{c_\gamma}},\ \widetilde{L}\bigg\}=: \widetilde{L}_1,
\end{align}
in the domain $\Omega_d$, where the constant $\widetilde{C}_1$ is defined in \eqref{2.24}. We first note by \eqref{2.24} and \eqref{2.16} that
\begin{align*}
\hat\alpha(r,0),\hat\beta(r,0)\leq (h(r,0))^{-\fr{2}{\gamma-1}}\mathcal{C}_1 = \fr{\gamma_k \mathcal{C}_1}{\rho_0(r)}e^{-\fr{S_0(r)}{c_\gamma}} <\fr{\gamma_k\mathcal{C}_1}{\dps\min_{r\in[b_1,b_2]}\rho_0(r)} e^{\fr{\mathcal{S}_0}{c_\gamma}}\leq \widetilde{L}_1,
\end{align*}
which means that \eqref{4.31} holds at the initial time. Without loss of generality, suppose that there exists a first time point $(r',t')\in\Omega_d$ such that $\hat\alpha(r',t')=\widetilde{L}_1$ and $\hat\alpha(r,t)<\widetilde{L}_1$, $\hat\beta(r,t)\leq\widetilde{L}_1$ for $(r,t)\in\Omega_d\cap\{t<t'\}$. Then there must hold
\begin{align}\label{4.32}
\pa_3\hat\alpha(r',t')\geq0.
\end{align}
However, one applies \eqref{4.27}-\eqref{4.30} to find that
\begin{align*}
\pa_3 \hat{\alpha}(r',t') =& \bigg\{\Big(\fr{\gamma+1}{4}h^{\fr{2}{\gamma-1}}e^{Mt}\hat{\alpha} +A_2\Big)(\hat\beta-\hat\alpha) -(A_1-A_2)\hat\beta \\
&\ +\Big(-\fr{h(\gamma u-h)}{\gamma c_\gamma c_1}r^m\rho \widetilde{S}_\xi +\fr{mu^2}{rc_1}-M\Big)\hat\beta +h^{-\fr{2}{\gamma-1}}e^{-Mt}A_3\bigg\}(r',t') \\
< & -\widetilde{L}_1 +e^{-Mt'}\widetilde{L}_1<0.
\end{align*}
Here we used the result $A_1-A_2>0$ in \eqref{4.13}. The above inequality contradicts \eqref{4.32}, and then \eqref{4.31} is proved.

We combine \eqref{4.26} and \eqref{4.31} to arrive at \eqref{4.24}. The proof of the lemma is fulfilled.
\end{proof}

\section{Proof of the theorems}\label{S5}

In this section, we contribute to the proof of Theorems \ref{thm1} and \ref{thm2}.

\subsection{Proof of Theorem \ref{thm1}}\label{S51}

We first derive the positive lower bound of $h$ in the domain $\Omega_d$.
\begin{lem}\label{lem8}
Let the assumptions in Lemma \ref{lem5} hold.
Then the smooth solution $(\rho, u, S)$ of \eqref{1.1}, \eqref{1.2} in the domain $\Omega_d$ satisfies
\begin{align}\label{5.1}
h(r,t)\geq \min_{r\in[b_1,b_2]}h_0(r)\exp\bigg\{-\fr{\gamma-1}{2} \bigg(\widetilde{\mathcal{C}}_1+\fr{2m\tilde{w}}{(3-\gamma)b_1}\bigg)T_0\bigg\},
\end{align}
where the constant $\widetilde{\mathcal{C}}_1$ is given in \eqref{4.18}.
\end{lem}
\begin{proof}
We rewrite the last equation of \eqref{a4} as
\begin{align}\label{5.2}
\pa_3\Big(h^{\fr{2}{\gamma-1}}\Big)=& -h^{\fr{2}{\gamma-1}}\bigg(\beta
    -\frac{h(\gamma u-h)}{\gamma c_\gamma c_1}r^m\rho \widetilde{S}_\xi
    +\frac{mu^2}{rc_1}\bigg) \notag \\
\geq & -h^{\fr{2}{\gamma-1}}\bigg(\beta +\frac{mu^2}{rc_1}\bigg).
\end{align}
For any point $(r,t)\in\Omega_d$, we draw the $3$-characteristic $r=r_3(t)$ through the point $(r,t)$ up to the line $t=0$. Integrating \eqref{5.2} along $r=r_3(t)$ and utilizing \eqref{4.18} yields
\begin{align}\label{5.3}
h^{\fr{2}{\gamma-1}}\geq & (h_0(r_3(0)))^{\fr{2}{\gamma-1}}\exp\bigg\{ -\int_{0}^t\bigg(\beta +\frac{mu^2}{rc_1}\bigg)(r_3(t),t)\ {\rm d}t\bigg\} \notag \\
\geq & \Big(\min_{r\in[b_1,b_2]}h_0(r)\Big)^{\fr{2}{\gamma-1}}\exp\bigg\{ -\bigg(\widetilde{\mathcal{C}}_1+\fr{2m\tilde{w}}{(3-\gamma)b_1}\bigg)T_0\bigg\},
\end{align}
which directly leads to \eqref{5.1}. The proof of the lemma is finished.
\end{proof}

\textbf{Proof of Theorem \ref{thm1}:} The proof is based on the classical framework of Li \cite{LiBook} by extending the local smooth solution to global domain. The local existence of smooth solutions can be directly obtained by the classical theory of quasilinear hyperbolic systems, see e.g. Li and Yu \cite{Li-Yu}. Moreover, the local existence time $\delta$ of the smooth solution depends only on the $C^1$ norm of the solution and the lower bound of $h$, see Remark 4.1. in Chapter 1 in \cite{Li-Yu}. According to the estimates in Lemmas \ref{lem2}, \ref{lem3}, \ref{lem5} and \ref{lem8}, we know that the local existence time $\delta$ is a constant. Thus we can solve a finite number of
local existence problems to extend the solution in the domain $\Omega_d\cap\{0\leq t\leq \delta\}$ to the entire domain $\Omega_d$. Finally, the properties of solution in \eqref{2.27} follow directly from the estimates in Lemmas \ref{lem2}, \ref{lem3} and \ref{lem5}. Therefore, we have completed the proof of Theorem \ref{thm1}.

\subsection{Proof of Theorem \ref{thm2}}\label{S52}

We first have the positive lower bound of $h$ in the new scene.
\begin{lem}\label{lem9}
Let Assumptions \ref{asu1} and \ref{asu2} hold.
Then the smooth solution $(\rho, u, S)$ of \eqref{1.1}, \eqref{1.2} in the domain $\Omega_d$ satisfies
\begin{align}\label{5.4}
h(r,t)\geq \underline{h}\ ,
\end{align}
where
$$
\underline{h}=\min_{r\in[b_1,b_2]}h_0(r)\exp\bigg\{-\fr{\gamma-1}{2} \bigg(\widetilde{\mathcal{C}}_2h_{d}^{\fr{2}{\gamma-1}}+\fr{2m\tilde{w}}{(3-\gamma)b_1}\bigg)T_0\bigg\},
$$
and the constant $\widetilde{\mathcal{C}}_2$ is given in \eqref{4.25}.
\end{lem}
\begin{proof}
The proof is similar to that of Lemma \ref{lem8}, so we omit it here.
\end{proof}

To show the singularity formation of the smooth solution, we take $\lambda=\fr{3-\gamma}{2(\gamma-1)}$ in \eqref{4.23} to acquire
\begin{align}\label{5.5}
\begin{split}
\pa_1 \tilde\beta =&-\dfrac{\gamma+1}{4}h^{\fr{3-\gamma}{2(\gamma-1)}} \tilde{\beta}^{2}
   +\Big(-B_{1}+\fr{(3-\gamma)h(\gamma u+h)}{4\gamma c_\gamma c_3}r^m\rho \widetilde{S}_\xi   +\fr{(3-\gamma)mu^2}{4rc_3} \Big)\tilde\beta \\
&\ + B_{2} \tilde\alpha +h^{-\fr{3-\gamma}{2(\gamma-1)}}B_3, \\[4pt]
\pa_3 \tilde\alpha =&-\dfrac{\gamma+1}{4}h^{\fr{3-\gamma}{2(\gamma-1)}} \tilde{\alpha}^{2}
  +\Big(-A_{1}-\fr{(3-\gamma)h(\gamma u-h)}{4\gamma c_\gamma c_1}r^m\rho \widetilde{S}_\xi  +\fr{(3-\gamma)mu^2}{4rc_1} \Big)\tilde\alpha \\
&\ + A_{2} \tilde\beta +h^{-\fr{3-\gamma}{2(\gamma-1)}}A_3,
\end{split}
\end{align}
where
\begin{align}\label{5.6}
\tilde\alpha =h^{-\fr{3-\gamma}{2(\gamma-1)}}\alpha,\qquad \tilde\beta =h^{-\fr{3-\gamma}{2(\gamma-1)}}\beta.
\end{align}
\begin{rem}\label{r7}
The two equations in \eqref{5.5} are decoupled in their leading quadratic order terms, which makes it convenient for us to control the right-hand side terms.
\end{rem}
In addition, note that
$$
\tilde\alpha=h^{\fr{\gamma+1}{2(\gamma-1)}}\cdot h^{-\fr{2}{\gamma-1}}\alpha,\quad \tilde\beta=h^{\fr{\gamma+1}{2(\gamma-1)}}\cdot h^{-\fr{2}{\gamma-1}}\beta,
$$
from which and \eqref{4.24} yields the upper bounds of $(\tilde\alpha, \tilde\beta)$
\begin{align}\label{a7}
\tilde\alpha< h_{d}^{\fr{\gamma+1}{2(\gamma-1)}}\widetilde{\mathcal{C}}_2,\qquad \tilde\beta< h_{d}^{\fr{\gamma+1}{2(\gamma-1)}}\widetilde{\mathcal{C}}_2,
\end{align}
provided that Assumptions \ref{asu1} and \ref{asu2} are fulfilled. Here the constant $\widetilde{\mathcal{C}}_2$ is given in \eqref{4.25}.

\textbf{Proof of Theorem \ref{thm2}:} We first estimate the right-hand side terms of \eqref{5.5}. Similar to \eqref{4.28} and \eqref{4.29}, one acquires
\begin{align}\label{5.7}
\begin{split}
\bigg|\fr{(3-\gamma)h(\gamma u+h)}{4\gamma c_\gamma c_3}r^m\rho \widetilde{S}_\xi   +\fr{(3-\gamma)mu^2}{4rc_3}\bigg|\leq & \fr{3-\gamma}{4}M, \\
\bigg|-\fr{(3-\gamma)h(\gamma u-h)}{4\gamma c_\gamma c_1}r^m\rho \widetilde{S}_\xi  +\fr{(3-\gamma)mu^2}{4rc_1}\bigg| \leq & \fr{3-\gamma}{4}M,
\end{split}
\end{align}
where the constant $M$ is defined in \eqref{a6}. We recall the expressions of $A_i, B_i$ in \eqref{3.2}, \eqref{3.3} and apply \eqref{4.11}-\eqref{a2} to obtain
\begin{align}\label{5.8}
0<A_1,B_1\leq & \fr{mc_3}{rc_{1}^2}\Big(\fr{\gamma-1}{2}u^2-h^2\Big)+\fr{2(3-\gamma)mu^2h^2}{rc_{1}^2c_3} \notag \\
\leq & \fr{mc_3}{rc_{1}^2}\cdot c_1c_3+\fr{2(3-\gamma)mu}{r}\Big(\fr{h}{c_1}\Big)^2\fr{u}{c_3} \notag \\
\leq & \fr{4m\tilde{w}}{(3-\gamma)b_1}+\fr{2(\gamma-1)^2m\tilde{w}}{(3-\gamma)b_1}< \fr{12m\tilde{w}}{(3-\gamma)b_1},
\end{align}
and
\begin{align}\label{5.9}
0<A_2,B_2\leq & \fr{mc_3}{2rc_{1}^2}\Big(\fr{\gamma-1}{2}u^2-h^2\Big) +\fr{h[3\gamma u^2+(\gamma+5)uh-h^2]}{4\gamma c_v c_1c_3}r^m\rho \widetilde{S}_\xi \notag \\
\leq & \fr{mc_3}{2rc_{1}^2}\Big(\fr{\gamma-1}{2}u^2-h^2\Big) +\fr{mc_3}{2rc_{1}^2}\Big(\fr{\gamma-1}{2}u^2-h^2\Big) +\fr{(3-\gamma)mu^2h^2}{rc_{1}^2c_3} \notag \\ < & \fr{8m\tilde{w}}{(3-\gamma)b_1}.
\end{align}
Furthermore, it concludes by \eqref{4.30} that
\begin{align}\label{5.10}
0<&h^{-\fr{3-\gamma}{2(\gamma-1)}}A_3,h^{-\fr{3-\gamma}{2(\gamma-1)}}B_3 \notag \\
\leq & h^{\fr{\gamma+1}{2(\gamma-1)}}\cdot h^{-\fr{2}{\gamma-1}}\bigg\{\fr{u^3h^2}{\gamma c_{v}^2 r c_{1}^2 c_{3}}mc_v\gamma \cdot r^m\rho \widetilde{S}_\xi +\fr{uh^2}{\gamma c_v c_1}r^{2m}\rho^2|\widetilde{S}_{\xi\xi}|\bigg\} \notag \\
\leq & h^{\fr{\gamma+1}{2(\gamma-1)}} \widetilde{L}< h_{d}^{\fr{\gamma+1}{2(\gamma-1)}} \widetilde{L},
\end{align}
where the constant $\widetilde{L}$ is given in \eqref{a6}.

Combining \eqref{a7}-\eqref{5.10}, we deduce by \eqref{5.5} and \eqref{5.4}
\begin{align}\label{5.11}
\begin{split}
\pa_1 \tilde\beta <&-\dfrac{\gamma+1}{8}{\underline{h}}^{\fr{3-\gamma}{2(\gamma-1)}} \tilde{\beta}^{2}
   +\Phi, \\
\pa_3 \tilde\alpha <&-\dfrac{\gamma+1}{8}{\underline{h}}^{\fr{3-\gamma}{2(\gamma-1)}} \tilde{\alpha}^{2}
  +\Psi,
\end{split}
\end{align}
where
\begin{align}\label{5.12}
\begin{split}
\Phi=&-\dfrac{\gamma+1}{8}{\underline{h}}^{\fr{3-\gamma}{2(\gamma-1)}} \tilde{\beta}^{2} +\Big(\fr{12m\tilde{w}}{(3-\gamma)b_1} + \fr{3-\gamma}{4}M \Big)|\tilde\beta| \\ &\ +\fr{8m\tilde{w}}{(3-\gamma)b_1}h_{d}^{\fr{\gamma+1}{2(\gamma-1)}}\widetilde{\mathcal{C}}_2 +h_{d}^{\fr{\gamma+1}{2(\gamma-1)}} \widetilde{L}, \\
\Psi=&-\dfrac{\gamma+1}{8}{\underline{h}}^{\fr{3-\gamma}{2(\gamma-1)}} \tilde{\alpha}^{2} +\Big(\fr{12m\tilde{w}}{(3-\gamma)b_1} + \fr{3-\gamma}{4}M \Big)|\tilde\alpha| \\ &\ +\fr{8m\tilde{w}}{(3-\gamma)b_1}h_{d}^{\fr{\gamma+1}{2(\gamma-1)}}\widetilde{\mathcal{C}}_2 +h_{d}^{\fr{\gamma+1}{2(\gamma-1)}} \widetilde{L}.
\end{split}
\end{align}
Here we used the non negativity of $A_2$ and $B_2$. It is evident by \eqref{5.12} that there exists a positive constant $\mathcal{N}_1$ such that $\Phi\leq0$ ($\Psi\leq0$) if $\tilde\beta\leq -\mathcal{N}_1$ ($\tilde\alpha\leq -\mathcal{N}_1$). If $\Phi\leq0$ on a 1-characteristic $r_1(t)$ or $\Psi\leq0$ on a 3-characteristic $r_3(t)$, one can easily employ \eqref{5.11} to estimate the blowup time
\begin{align*}
T^*=\fr{8}{-(\gamma+1)\tilde\beta(r_1(0),0)}\underline{h}^{-\fr{3-\gamma}{2(\gamma-1)}}\quad {\rm or}\quad \fr{8}{-(\gamma+1)\tilde\alpha(r_3(0),0)}\underline{h}^{-\fr{3-\gamma}{2(\gamma-1)}}.
\end{align*}
To ensure $T^*<T\leq T_0$, we only need the initial data to meet the following
\begin{align}\label{5.13}
\tilde\beta(r^*,0)\leq -\fr{8}{(\gamma+1)T}\underline{h}^{-\fr{3-\gamma}{2(\gamma-1)}}, \qquad {\rm or}\qquad \tilde\alpha(r^*,0)\leq -\fr{8}{(\gamma+1)T}\underline{h}^{-\fr{3-\gamma}{2(\gamma-1)}},
\end{align}
for some number $r^*\in(b_1,b_2)$.

According to the above analysis, we now assume that the initial $\alpha_0(r)$ or $\beta_0(r)$ is very negative at a point $r^*\in(b_1,b_2)$ such that
\begin{align}\label{5.14}
\alpha_0 (r^*)\leq -\mathcal{N},\qquad {\rm or}\qquad \beta_0 (r^*) \leq -\mathcal{N},
\end{align}
where
\begin{align}\label{5.15}
\mathcal{N}=\max\bigg\{\fr{8}{(\gamma+1)T} \Big(\fr{h_d}{\underline{h}}\Big)^{\fr{3-\gamma}{2(\gamma-1)}},\ h_{d}^{\fr{3-\gamma}{2(\gamma-1)}}\mathcal{N}_1\bigg\}.
\end{align}
Then it suggests by \eqref{5.6} and \eqref{5.14} that
\begin{align}\label{5.16}
\tilde\alpha (r^*,0)\leq -\max\bigg\{\fr{8{\underline{h}}^{-\fr{3-\gamma}{2(\gamma-1)}}}{(\gamma+1)T},\ \mathcal{N}_1\bigg\},\quad {\rm or}\quad \tilde\beta (r^*,0) \leq -\max\bigg\{\fr{8{\underline{h}}^{-\fr{3-\gamma}{2(\gamma-1)}}}{(\gamma+1)T} ,\ \mathcal{N}_1\bigg\}.
\end{align}
We assert that, before the blowup occurs, there always hold
\begin{align}\label{5.17}
\tilde\alpha(r_3(t;r^*,0),t)\leq -\mathcal{N}_1,\qquad {\rm or}\qquad \tilde\beta(r_1(t;r^*,0),t) \leq -\mathcal{N}_1,
\end{align}
where $r_3(t;r^*,0),t$ and $r_1(t;r^*,0)$ are the 3- and 1-characteristics through the point $(r^*,0)$, respectively. Without loss of generality, let $\alpha_0 (r^*)\leq -\mathcal{N}$ be assumed and accordingly the first equation of \eqref{5.16} is satisfied. In the light of the source of $\mathcal{N}_1$ before, we see that $\Psi(r*,0)\leq0$, which together with \eqref{5.11} yields
$$
\pa_3\tilde\alpha(r^*,0) <0.
$$
Thus $\tilde\alpha(r_3(t;r^*,0),t)<\mathcal{N}_1$ for small $t>0$. Suppose that there exists a first time $t'>0$ such that $\tilde\alpha(r_3(t';r^*,0),t')=\mathcal{N}_1$ and $\tilde\alpha(r_3(t;r^*,0),t)<\mathcal{N}_1$ for $t<t'$. Then there must have
$$
\pa_3\tilde\alpha(r_3(t';r^*,0),t')\geq0.
$$
On the other hand, it is note by Lemma \ref{lem7} that the upper bound estimates in \eqref{a7} remain valid until the blowup occurs. Hence one acquires for $t\in[0,t']$
$$
\Psi(r_3(t;r^*,0),t)\leq 0,
$$
which implies by \eqref{5.11} that
\begin{align*}
\pa_3 \tilde\alpha(r_3(t';r^*,0),t') <0,
\end{align*}
a contradiction, which proves the above assertion. Based on the selection of $\mathcal{N}$ in \eqref{5.15}, we know that the singularity forms no later than $T$.

In sum, we have completed the proof of Theorem \ref{thm2}.

\section*{Acknowledgements}

The first and second authors were partially supported by National Science Foundation (DMS-2008504, DMS-2306258), the third author was partially supported by National Natural Science Foundation of China (12171130) and Natural Science Foundation of
Zhejiang province of China (LMS25A010014).

\section*{Statements and Declarations}
The authors declare that they have no conflict of interest.

\section*{Data Availability Statement}
No data was used for the research described in the paper.

\end{document}